\documentclass[english,a4paper,11pt]{amsart}
\usepackage[latin1]{inputenc}
\usepackage{amssymb,amsthm}
\usepackage[all]{xy}
\usepackage{amsmath}
\usepackage{indentfirst}
\usepackage{amsfonts}
\usepackage{textcomp}
\usepackage{graphicx}
\usepackage{enumerate}
\usepackage{hyperref}
\newtheorem{lemma}{Lemma}[section]
\newtheorem{theorem}[lemma]{Theorem}

\newtheorem{proposition}[lemma]{Proposition}
\newtheorem{corollary}[lemma]{Corollary}
\theoremstyle{definition}
\newtheorem{definition}[lemma]{Definition}
\theoremstyle{remark}
\newtheorem{remark}[lemma]{Remark}

\newtheorem{example}[lemma]{Example}

\newtheorem{question}[lemma]{Question}

\newcommand{\kod}{\kappa}
\newcommand{\vol}{\textrm{vol}}

\newcommand{\Pic}{\textrm{Pic}}

\newcommand{\Q}{\mathbb{Q}}
\newcommand{\Z}{{\mathbb Z}}
\newcommand{\R}{\mathbb{R}}

\newcommand{\N}{\mathbb{N}}
\newcommand{\C}{\mathbb{C}}
\newcommand{\Supp}{\textrm{Supp}}

\newcommand{\supp}{\mbox{Supp}}

\newcommand{\ord}{\mbox{ord}}

\def\ord{\operatorname{ord}}

\def\supp{\text{Supp}}

\def\codim{\mathrm{codim}}

\def\Z{{\mathbb Z}}
\def\N{{\mathbb N}}

\def\R{{\mathbb R}}
\def\Q{{\mathbb Q}}
\def\C{{\mathbb C}}

\def\Bs{\mathop{\rm Base}\nolimits}

\def\codim{\mathop{\rm codim}\nolimits}

\def\Im{\mathop{\rm Im}\nolimits}

\def\Pic{\mathop{\rm Pic}\nolimits}

\def\supp{\mathop{\rm Supp}\nolimits}

\def\vol{\mathop{\rm vol}\nolimits}

\def\tilde{\widetilde}

\def\phi{\varphi}

\def\k{{\kappa}}

\def\cJ{{\mathcal J}}

\def\cO{{\mathcal O}}

\def\ord{\mathop{\rm ord}\nolimits}

\def\ge{\geqslant}

\theoremstyle{theorem}

\begin{document}

\title{Restricted volumes of effective divisors}
\thanks{{\it 2010 Math classification:} 14C20, 14F18}
\thanks{\today}
\author{Lorenzo Di Biagio and Gianluca Pacienza}
\address{Instytut Matematyki Uniwersytetu Warszawskiego, ul. S. Banacha 2, 02-097 Warszawa -- Poland }
\email{lorenzo.dibiagio@gmail.com}
\address{ IRMA, Universit\'e de Strasbourg et CNRS, 7 rue R. Descartes, 67084 Strasbourg Cedex -- France}
\email{pacienza@math.unistra.fr }

\begin{abstract}
We study the restricted volume of effective divisors, its properties and the relationship with the related notion of reduced volume, defined via multiplier ideals, and with the asymptotic intersection number. We build upon the fundamental work of Lazarsfeld and Must\u a\c ta relating the restricted volume of big divisors to the volume of the associated Okounkov body. We extend their constructions and results to the case of effective divisors, recovering some results of Kaveh and Khovanskii, proving a Fujita-type approximation in this larger setting and studying the restricted volume function. In order to relate the reduced volume and the asymptotic intersection number we investigate a boundedness property of asymptotic multiplier ideals and prove it holds, for instance, for finitely generated divisors. In this way we obtain also a complete picture for the canonical divisor of an arbitrary smooth projective variety and for nef divisors on varieties of dimension at most 3.
\end{abstract}

\maketitle


\setcounter{section}{0}
\setcounter{lemma}{0}

%
\section{Introduction}
%

%
\subsection{Motivations}
%
Let $X$ be a smooth complex projective variety of dimension $n$ and let $D$ be a Cartier divisor. As it is well known, ample divisors display beautiful geometric properties, as well as cohomological and numerical ones. For long time it was thought that little, instead, could be said about general effective divisors. Anyway, at the beginning of this new century it became clear (see, e.g., \cite{Nakayama}, \cite{LazI}, \cite{ELMNPbundles}) that the classes of big divisors share some of these properties if one is willing to work asymptotically: if $D$ is big, usually $D^n$ does not carry immediate geometric information, while it is the \emph{volume} of $D$, $\vol(D):=\lim_{m \rightarrow +\infty } \frac{h^0(X,\mathcal{O}_X(mD))}{m^n/n!}$, that  is actually the right generalization of the top intersection number of ample divisors.

Let now $V$ be a $d$-dimensional irreducible subvariety of $X$. If $A$ is ample then the intersection number $A^d \cdot V$ plays an important r\^ole in many geometric questions. Along the previous lines, Ein \emph{et al.} started in \cite{ELMNPrestricted} a thorough study of asymptotic analogues of this degree for big divisors on $V$. \\
In case of an arbitrary divisor $D$, the \emph{restricted volume} of $D$ along $V$ is $$\vol_{X|V}(D):=\limsup_{m \rightarrow + \infty} \frac{\dim  H^0(X|V,\mathcal{O}_X(mD))}{m^d/d!},$$
where $H^0(X|V,\mathcal{O}_X(mD)):= \mathrm{Im} \left( H^0(X,\mathcal{O}_X(mD)) \stackrel{\mathrm{restr}_V}{\rightarrow} H^0(V,\mathcal{O}_V(mD)) \right) $.
In the big case, restricted volumes  have played an important r\^ole in the proof of boundedness of pluricanonical maps of varieties of general type (see \cite{HcMcK}, \cite{Takayama}), in the Fujita type results (see, e.g., \cite[Theorem 2.20]{ELMNPrestricted} and \cite[Proposition 5.3]{Takayama}), as well as in the study of the cone of pseudoeffective divisors \cite{BDPP}.

For an ample divisor $A$, since by Serre's vanishing theorem the restriction maps are eventually surjective, we have 
\begin{equation} \label{eq:ampio}
\vol_{X|V}(A)= \vol_V (A_{|V})=A^d \cdot V.
\end{equation}
For an arbitrary divisor $D$ on $X$,  different generalizations of the intersection number have been studied in the recent literature (cf. \cite{ELMNPrestricted} and \cite{Takayama}): apart from $\vol_{X|V}(D)$, the \emph{reduced volume} $\mu(V,D)$ of $D$ along $V$, and  the \emph{asymptotic intersection number} $\|D^d\cdot V\|$ of $D$ and $V$:
$$\mu(V,D) := \limsup_{m \rightarrow + \infty} \frac{\dim\left( H^0(\mathcal{O}_V(mD) \otimes \mathcal{J}(X, \|mD\|)_V \right)}{m^d/d!},$$ 
$$\|D^d \cdot V\|= \limsup_{m \rightarrow + \infty} \frac{\sharp \left( V \cap D_{m,1} \cap \dots \cap D_{m,d} \setminus \mathrm{Bs}(|mD|)\right)}{m^d/d!},$$ 
where $\mathcal{J}(X,\|mD\|)_V$ is the ideal in $\mathcal{O}_V$ generated by the asymptotic multiplier ideal $\mathcal{J}(X, \|mD\|)$ (see \cite[Definition 11.1.2]{LazII}) and $D_{m,1}, \dots, D_{m,d}$ are general elements in $|mD|$. (See Section \ref{sect:asymptoticintersection} and Section \ref{sect:volume}).

If $D$ is big and $V$ does not lie in a closed subset of $X$ (which is precisely the augmented base locus of $D$, cf. \cite{ELMNP} for the definition) then we have  that $\vol_{X|V}(D)=\mu(V,D)=\|D^d \cdot V\|$ (see \cite[Theorem 2.13]{ELMNPrestricted} and \cite[Theorem 3.1]{Takayama}). In particular the equality $\vol_{X|V}(D)=\|D^d\cdot V\|$ can be read as a generalization of Fujita's approximation theorem, and it leads to several interesting consequences (see \cite[Corollary 2.14, Corollary 2.15, Corollary 2.16]{ELMNPrestricted}).  When the divisor $D$ is  effective, but not big, the lack of positivity seemed to be both a technicle obstacle and a possible source of pathologies. 

The second named author and Takayama have initiated, in \cite{PT}, the study of the notions of restricted, reduced volume and asymptotic intersection number for line bundles $L$ such that $0<\k(X,L)<\dim(X)$. Their main achievement is the proof that if $V$ contains a (very) general point then $\mu(V,L)>0\Leftrightarrow \vol_{X|V}(L)>0$. These two positivities are in turn equivalent to requiring that 
\begin{align}  \label{eqn:iitaka}
& \text{the Iitaka fibration $f$ associated to $L$ and restricted to $V$ gives a generically finite map} \\
& \text{$f_{|V}: V \dashrightarrow f(V)$}  \nonumber
\end{align}
(see \cite[Theorem 1.1, Corollary 1.2]{PT}). For such subvarieties the rate of growth of the dimensions of the restricted linear series is maximal and we are therefore in a situation similar to the big case.  


The purpose of this paper is 
to show that indeed restricted volumes of effective divisors along subvarieties verifying condition (\ref{eqn:iitaka}) (which turns out to be an optimal condition) share the same nice properties enjoyed by restricted volumes of big divisors. Moreover, under the same condition, the relationships among the different variants are well understood, modulo a property, which holds for instance for finitely generated divisors on arbitrary varieties (hence for the canonical divisor) or in low dimension. 

%
\subsection{The contents of the paper}
%
In the first part of the present paper we concentrate on the volume function of graded linear series associated to effective divisors. We study its properties and its relationship with the asymptotic intersection number. Lazarsfeld and Must\u a\c ta \cite{LM} and, independently, Kaveh and Khovanskii \cite{KK} developed a powerful approach, initiated by Okounkov \cite{O1,O2}, to study volumes of divisors and, more generally, linear series via convex geometry, by associating to them a convex body, called the {\it Okounkov body}. Their similar viewpoints lead to some overlaps as well as to several differences. 
On the one hand \cite{KK} deals more generally with linear series associated to effective divisors, while in \cite{LM} the theory is worked out for big divisors (or for linear series inducing a birational map onto the image). On the other hand, the introduction of a {\it global} Okounkov body, done in \cite{LM}, renders particularly transparent many of the properties of the volume function  and allows to study the way it varies in a family. 
In  the present paper we show that for complex projective varieties the approach adopted in 
\cite{LM} works as well to deal with linear series associated to effective divisors (so that one can recover some of the main results of \cite{KK}) and, at the same time, it still allows, in this larger framework, to understand how the volume function behave when we let the effective divisors and the linear series vary.
 
Then we turn to the relationship between the reduced volume and the asymptotic intersection number.  We underline the importance of a property in such a study, namely the boundedness of the difference between the base and the asymptotic multiplier ideals of a graded linear series. We exhibit two classes  of linear series for which the property holds and show how the property implies the equality between the reduced volume and the asymptotic intersection number. We wonder whether the boundedness property holds for any big graded linear series.

Putting together the results of the two parts we obtain, in particular, a complete understanding of the relationships between the three types of volumes along subvarieties for the canonical divisor of a smooth complex projective variety. 

%
\subsection{The statements of the results}
%

From now on all varieties are assumed to be complex, reduced and irreducible.

Recall that a graded linear series $W_\bullet$ associated to a Cartier divisor $D$  
(or line bundle) on a variety $V$ is the datum, for every $m\in\N$, of a subspace $W_m\subseteq H^0(V,mD)$ such that $W_0 = \mathbb{C}$ and $W_k\cdot W_l\subseteq W_{k+l}$. If $d=\dim(V)$ the volume of $W_\bullet$ is 
$$
 \vol_V(W_\bullet):= \limsup_{m \rightarrow + \infty} \frac{\dim  (W_m)}{m^{d}/d!}.
$$

The guiding example is provided by graded linear series obtained by restricting sections, i.e. $W_m:=H^0(X|V,mD)$.

To any $W_\bullet$ as before, given an admissible flag (see \S \ref{subsect:okounkov}), we can attach its Okounkov body, i.e., a closed convex set $\Delta(W_\bullet) \subseteq \mathbb{R}^d$. If $V$ is projective, $\Delta(W_\bullet)$ is also compact.


\begin{theorem} \label{thm:volumeA}
Let $V$ be a projective variety of dimension $d$ and let $W_{\bullet}$ be a graded linear series associated to a line bundle $L$. Then there exists an admissible flag $F_{\bullet}$ such that $$\vol_{\mathbb{R}^d}(\Delta(W_{\bullet}))=\frac{1}{d!} \vol(W_{\bullet}).$$
\end{theorem}

 \begin{theorem} [Fujita's approximation] \label{thm:fujita}
Let $X$ be a smooth projective variety and let $L \in \Pic(X)$.  Let $V \subseteq X$ be a $d$-dimensional subvariety. Suppose that $\vol_{X|V}(L)>0$. Then for every $\epsilon > 0$ there exists a birational morphism $\mu_\epsilon: X_\epsilon \rightarrow X$ of smooth varieties that is an isomorphism over the generic point of $V$, and a decomposition  $\mu_\epsilon^*(L)=A_\epsilon+E_\epsilon$ such that 
\begin{enumerate}[(i)]
\item \label{tesi1} $E_{\epsilon}$ is an effective $\mathbb{Q}$-divisor such that $\Supp(E_\epsilon)$ does not contain the strict transform $V_\epsilon$ of $V$; 
\item \label{tesi2} $A_\epsilon$ is a $\mathbb{Q}$-divisor that has an integral multiple that is base point free;

\item \label{tesi3}
 $$ \vol_{X|V}(L)\geq \vol_{X_\epsilon|V_\epsilon}(A_{\epsilon}) \geq \vol_{X|V}(L) - \epsilon',$$ where $\epsilon' = d! \cdot \epsilon$.\\ Equivalently, setting $\delta= \min_m(\deg({\phi_{|mL|}}_{|V}: V \dashrightarrow \phi(V)))$, $$ \delta \vol_{X|V}(L)\geq (A_{\epsilon})^d\cdot V_\epsilon \geq \delta \vol_{X|V}(L) - \overline{\epsilon},$$ where $\overline{\epsilon}=d! \cdot \delta \epsilon$.
  \end{enumerate}
 Moreover 
 \begin{enumerate}
 \item[(iv)]
 $\delta \vol_{X|V}(L) = \|L^d \cdot V \|$.
  \end{enumerate}

  \end{theorem}
 

\begin{corollary}[Log-concavity of the restricted volume]\label{cor:log-concavity}
Let $X$ be a smooth projective variety and let $L_1,L_2\in \Pic(X)$
be two line bundles of non-negative Iitaka dimension. Let $V$ be a $d$-dimensional subvariety of $X$. Then the following holds:
$$ \vol_{X|V}(L_1+L_2)^{1/d}\geq \vol_{X|V}(L_1)^{1/d} + \vol_{X|V}(L_2)^{1/d}.
$$
\end{corollary}
Theorem \ref{thm:volumeA} and Corollary \ref{cor:log-concavity} correspond, respectively, to \cite[Corollary 3.11]{KK}, item 1) and 3). 
Theorem \ref{thm:fujita} is proved using Theorem \ref{thm:fujitaseries}, the latter corresponding to \cite[Corollary 3.11]{KK}, item 2).

From Theorem \ref{thm:fujita}, item (iv),
we can also deduce the following corollary:

\begin{corollary}\label{cor:caratt}
Let $X$ be a smooth projective variety and let $L$ be a line bundle. Let $x \in X$ be a general point. Fix $d$ between $1$ and $\dim(X)-1$. Assume that $\mu(V,L)=\vol_{X|V}(L)$ for every $d$-dimensional subvariety passing through $x$. Then either $L$ is big or $\kod(X,L) < d$.
\end{corollary}

Then we move forward to study how the Okounkov bodies of graded linear series associated to effective divisors vary in family. The set-up is the following. Let $V$ be a variety of dimension $d$ and let $D_1,\ldots, D_r$ be (effective) Cartier divisors on $V$. For any 
$\vec{m}=(m_1,\ldots,m_r)\in \N^r$ we write 
$$
 \vec m D:=m_1D_1+\ldots +m_rD_r.
$$
\begin{definition}
A multi-graded linear series $W_{\vec \bullet}$ on $V$ associated to the divisors $D_1,\ldots,D_r$ is the data of finite-dimensional vector spaces
$$
 W_{\vec k}\subseteq H^0(V,\vec k D)
$$
for every $\vec k\in \N^r$, with $W_{\vec 0}=\C$, such that
$$
 W_{\vec k}\cdot W_{\vec m}\subseteq W_{\vec k+\vec m}. 
$$
 \hfill $\square$ \end{definition} 
Notice that to a multi-graded linear series $W_{\vec\bullet}$ and  integral vector $\vec a\in \N^r$ we can associate a single graded linear series $W_{\vec a, \bullet}$ 
as follows:
$
W_{\vec a, k}:=W_{k\vec a},\ \forall k\in\N.
$
We say that $W_{\vec\bullet}$ satisfies condition (GF') essentially if  $W_{k\vec a} \not = 0$ for all $k \gg 1$ and if one  
of the induced single graded linear series $W_{\vec a, \bullet}$ determines a generically finite map onto the image (see Definitions \ref{def:GF'} and \ref{def:condGF} for the details).

As in the big case \cite[Theorem B]{LM}, the following result explains geometrically  the regularity properties of the volume function in our more general framework.

\begin{theorem}\label{thm:infamiglia}
Let $V$ be a $d$-dimensional projective variety and $W_{\vec\bullet}$ as above. Suppose $W_{\vec\bullet}$ satisfies condition (GF'). Let $F_\bullet$ be an admissible flag as in Lemma \ref{lem:2}. There exists a closed convex cone depending on $F_\bullet$ 
$$
 \Delta(W_{\vec\bullet}) \subseteq \R^d\times \R^r
$$
such that
for all $\vec a\in {\mathrm{int}}(\supp(W_{\vec \bullet}))$, the fiber $(pr_2)_{|\Delta(W_{\vec \bullet})}^{-1}(\vec a)$ of $\Delta (W_\bullet)$ over $\vec a$ is the Okounkov body $\Delta(W_{\vec a, \bullet})$ of $W_{\vec a, \bullet}$.
\end{theorem}
From the theorem above one immediately deduces the following.
\begin{corollary}\label{cor:cont}
Let $V$ be a $d$-dimensional projective variety and $W_{\vec\bullet}$ as above.
Suppose $W_{\vec\bullet}$ satisfies condition (GF'). The function
$\vec a\mapsto \vol(W_{\vec a,\bullet})$ extends uniquely to a continuous function 
$$
 \vol_{W_{\vec \bullet}}: {\mathrm{int}} \big(\supp(W_{\vec \bullet}) \big)\to \R
$$
which is homogeneous of degree $d$, and the resulting function is log-concave.
\end{corollary}

Finally, using Theorems \ref{thm:fujita} and \ref{thm:infamiglia}, we are able to obtain (cf. Theorem \ref{thm:multigraded}) a generalized Fujita approximation for multi-graded linear series associated to effective divisors, generalizing the main result of \cite{Jow}. We refer the reader to \S \ref{ss:multigraded} for the relevant definitions and to Theorem \ref{thm:multigraded}  for  the precise statement. 

As anticipated in \S 1.2, in the second part of the paper we try to understand the relationship between the reduced volume and the asymptotic intersection number. 
The crucial property is the following. 

\begin{definition} 
Let $X$ be a smooth projective variety and $D$ a divisor of non-negative Iitaka dimension. We say that $D$ satisfies property $(\star)$ if there exists an effective divisor $N$ on $X$ such that, for any integer $p$ such that $h^0(X,pD)\not=0$, we have
\begin{equation} 
\cJ(X, \|pD\|)(-N)\subseteq \mathfrak{b}(|pD|),  \tag{$\star$}
\end{equation}
where $\cJ(X, \|pD\|)$ (resp. $\mathfrak{b}(|pD|)$) is the asymptotic multiplier ideal (resp. the base ideal) associated to $|pD|$ (see \cite[Definition 11.1.2]{LazII}).  
 \hfill $\square$ \end{definition}

The property ($\star$) can be enunciated also for (non-complete) graded linear series $W_\bullet$ (see \cite[Definition 1.1.24]{LazII} for the definition of asymptotic multiplier ideal of a graded linear series).

The birational version of the property above is called b-($\star$) (see Definition \ref{pinco}). 
If $D$ is a big divisor then it satisfies property (\ref{fondamentale}) (see \cite[Theorem 11.2.21]{LazII}). 
The importance of property b-$(\star)$ comes from the following result.

\begin{theorem}\label{thm:mu}
Let $X$ be a smooth projective variety and $L\in \Pic(X)$ be 
a line bundle verifying property b-($\star$). Then there exists a closed subset $\Lambda\subset X$ such that for every $d$-dimensional irreducible subvariety $V$ of $X$ such that $V\not\subseteq \Lambda$, we have: 

\begin{equation}\label{eq:mu}
\mu(V,L)=\|L^d\cdot V\|.
\end{equation}
\end{theorem}

Then, merging together Proposition \ref{goodmoriwaki} and Proposition \ref{contideali}, we exhibit two classes of divisors satisfying property b-$(\star)$.

\begin{theorem}\label{thm:esempi}
Let $X$ be a normal projective variety and $D$ a Cartier divisor on $X$.
\begin{enumerate}
\item[(i)] If $D$ is finitely generated, then it satisfies the property b-$(\star)$.
\item[(ii)] If $D$ is the pull-back of a big divisor, then it satisfies the property b-$(\star)$.
\end{enumerate}
\end{theorem}

Putting together Theorems \ref{thm:fujita}, \ref{thm:mu} and \ref{thm:esempi} (plus Remarks \ref{rmk:delta}, \ref{rmk:zero}) and using the finite generation of the canonical ring \cite{BCHM} (resp. some results on surfaces and 3-folds - see Proposition \ref{surfaces} and Proposition \ref{prop:3folds})  we obtain the following.

\begin{corollary}\label{cor:fine}
Let $X$ be a smooth projective variety  and $L\in \Pic(X)$ such that $\kod(X,L)\geq 0$. Suppose that one of the following three hypotheses holds
\begin{enumerate}[(i)]
\item $L=\omega_X$.
\item $\dim(X)=2$.
\item $L$ is nef and $\dim(X)=3$.
\end{enumerate}
Then there exists a closed subset $\Lambda\subset X$ such that for every $d$-dimensional irreducible subvariety $V$ of $X$ such that $V\not\subseteq \Lambda$, we have 
$$
 \mu(V,L)=\|L^d\cdot V\|=\delta\cdot \vol_{X|V}(L)
$$
where $\delta$ is the degree of the Iitaka fibration of $L$ restricted to $V$. 
\end{corollary}
As in the complete case, given a graded linear series $W_{\bullet}$ on a variety $V$, we can define $W_\bullet$ as \emph{big} if $\vol(W_\bullet)>0$. Notice that when $V$ is projective this is equivalent to requiring that there exists a positive integer $m$ for which $W_m$ gives a generically finite map onto the image (see Lemma \ref{lem:necessary}, Corollary \ref{cor:sufficient}). It is natural to ask the following.

\begin{question}\label{q:star} Let $V$ be a normal projective variety and $W_\bullet$ a graded linear series on it. Suppose that $W_\bullet$ is big. Does $W_\bullet$ satisfy property b-$(\star)$? \hfill $\square$
\end{question}
The consequences of a positive answer are discussed at the end of \S 5.

The paper is organized as follows. After some preliminaries, in \S 3 we recall the construction of the Okounkov body associated to a graded linear series  and prove all the results involving such construction, namely Theorems \ref{thm:volumeA}, \ref{thm:fujita}, \ref{thm:infamiglia} as well as Corollaries \ref{cor:log-concavity}, \ref{cor:caratt}, \ref{cor:cont}. In \S 4 we introduce and study property b-($\star$) and prove it holds for finitely generated divisors, pull-back of big divisors and in low dimension. Theorem \ref{thm:mu} is proved in \S 5. In the concluding section we investigate the relationship between property b-($\star$) and asymptotic valuations.

\subsection*{Acknowledgements}   This work started during the first named author's stay in Strasbourg and then continued during the first author's tenure of an ERCIM ``Alain Bensoussan'' Fellowship Programme. This Programme is supported by the Marie Curie Co-funding
of Regional, National and International Programmes (COFUND) of the European
Commission. He would like to thank the Mathematical Department of Strasbourg  for the kind hospitality; the University of Roma Tre and the French--Italian European Research Group in Algebraic Geometry GDRE-GRIFGA for their support. He would also like to thank the Mathematical Department of the University of Warsaw and Jaros\l aw Wi\'sniewski for their kind hospitality during the tenure of the ERCIM fellowship.  G.P. was partially supported by the ANR project ``CLASS'' no. ANR-10-JCJC-0111.
He wishes to thank S. Boucksom and C. Birkar for useful discussions.

\section{Preliminaries}
\subsection{Notation and conventions}
We will work over the field of complex numbers, $\mathbb{C}$. As in \cite{LazI}, \cite{LazII} a {\textit{scheme}} is a separated algebraic scheme of finite type over $\mathbb{C}$. A {\textit{variety}} is a reduced, irreducible scheme. A {\textit{ curve, surface, d-fold}} is a variety of dimension $1,2,d$, respectively. We will usually deal with closed points of schemes, unless otherwise specified. 
Given a variety $X$, by (Weil or Cartier) divisor we mean an integral divisor, unless otherwise stated. Given a Cartier divisor $D$ on $X$, we denote by $\kod(X,D)$ or simply $\kod(D)$ its Iitaka dimension.

A \emph{pair} $(X,\Delta)$ consists of a normal projective variety $X$ together with a $\mathbb{Q}$-Weil divisor $\Delta$ such that $K_X+\Delta$ is $\mathbb{Q}$-Cartier. 

Given a pair $(X, \Delta)$ and a Cartier divisor $D$ on $X$ we will denote by $\mathcal{J}((X,\Delta);\|D\|)$ the asymptotic multiplier ideal associated to $(X,\Delta)$ and $D$. See \cite[Definition 11.1.2, Remark 11.1.13]{LazII}. When $X$ is smooth and $\Delta=0$ we will just write $\mathcal{J}(X,\|D\|)$ or $\mathcal{J}(\|D\|)$.

Let $\mu: Y \rightarrow X$ be a morphism of normal varieties and let $\mathfrak{a} \subseteq \mathcal{O}_X$ be a sheaf of ideals on $X$. We will denote by $\mathfrak{a} \cdot \mathcal{O}_Y$ the inverse image ideal sheaf $\mu^{-1} \mathfrak{a} \cdot \mathcal{O}_Y$ (see \cite[p. 163]{Hartshorne}) and by $\overline{\mathfrak{a}}$ the integral closure of $\mathfrak{a}$ (see, for example, \cite[Definition 9.6.2]{LazII}).

\begin{definition} (see \cite[Definition 2.1.1, 2.1.20, 2.1.26]{LazI}).
Let $D$ be a Cartier divisor on a projective variety $X$. The \emph{semigroup} of $D$ is $$\mathbb{N}(D):=\{m \geq 0 | H^0(X,mD) \not = 0 \}.$$
If $\mathbb{N}(D) \not = \{0\}$ then there exist a sufficiently large $k$ and an $e=e(D) \geq 1$ such that $\mathbb{N}(D) \cap \{m : m \geq k\} = e\mathbb{N} \cap \{m : m \geq k\}$. We call this $e(D)$ the \emph{exponent} of $D$. Actually $e(D)$ is the gcd of all the elements in $\mathbb{N}(D)$. Note that if $D$ is big then $e(D)=1$.

Given $D$ as before, recall that the \emph{stable base locus} of $D$ is defined as $$\mathbb{B}(D):=\cap_{m \geq 1} \mathrm{Bs}(|mD|),$$ where $\mathrm{Bs}(|mD|)$ is the base locus of $|mD|$. It is defined only as a closed set, without any scheme structure, and it is the minimal element of the family of closed sets $\{\mathrm{Bs}(|mD|)\}_{m \geq 1}$. Moreover, there exists an integer $m_0$ such that $\mathbb{B}(D)=\mathrm{Bs}(|km_0D|)$ for all $k \gg 1$. Consider the non-zero semigroup $$\mathbb{M}(D):=\{m \geq 0 | \textrm{Bs}(|mD|)=\mathbb{B}(D) \}.$$ We will denote by $\overline{e}(D)$ the gcd of the elements in $\mathbb{M}(D)$. Notice that it can happen that $\overline{e}(D) > 1$ even if $e(D)=1$ (see \cite[Remark 2.1.22, Example 2.3.4]{LazI}).
 \hfill $\square$ \end{definition}

\begin{definition} (see \cite[Definition 2.4.1]{LazI}).
Given a Cartier divisor $D$ on a variety $X$, a \textit{graded linear series} $W_\bullet$ belonging to $D$ is a collection of finite dimensional vector spaces $W_m \subseteq H^0(X, mD)$ such that $W_0 = \mathbb{C}$ and $W_k \cdot W_l \subseteq W_{k+l}$ for all $k,l \geq 0$, where $W_k \cdot W_l$ denotes the image of $W_k \otimes W_l$ under the homomorphism $H^0(X,kD) \otimes H^0(X,lD) \rightarrow H^0(X, (k+l)D)$. Given $W_\bullet$ as before we can define its exponent $e(W_\bullet)$ as in the previous definition.  Given $W_\bullet$ as before, and a positive integer number $h$, we denote by $W_{h,\bullet}$ the graded linear series belonging to $hD$ and defined as $W_{h,m}:=W_{hm}$.
 \hfill $\square$ \end{definition}

\subsection{Birational morphisms and ideals}
\begin{lemma} \label{chiusura integrale}
Let $X$ be a normal projective variety and let $\mathfrak{a} \subseteq \mathcal{O}_X$ be a coherent sheaf of ideals. Let $\mu: Y \rightarrow X$ be any birational morphism from a normal variety $Y$ onto $X$. Then $\mathfrak{a} \subseteq \mu_*(\mathfrak{a} \cdot \mathcal{O}_Y) \subseteq \overline{\mathfrak{a}}$.
\end{lemma}
\begin{proof}
By a local analysis it is clear that $\mathfrak{a} \subseteq \mu_*(\mathfrak{a} \cdot \mathcal{O}_Y)$, hence we need only to show the second inclusion. Let $\mu': Z\rightarrow Y$ be the normalization of the blow-up of $Y$ along $\mathfrak{a} \cdot \mathcal{O}_Y$. Denoting by $E$ the exceptional divisor of $\mu'$, we have that $\mathfrak{a} \cdot \mathcal{O}_Z = (\mathfrak{a} \cdot \mathcal{O}_Y) \cdot \mathcal{O}_Z = \mathcal{O}_Z(-E)$. 
By definition $\mu'_*(\mathcal{O}_Z(-E))=\overline{\mathfrak{a} \cdot \mathcal{O}_Y} \supseteq \mathfrak{a} \cdot \mathcal{O}_Y$. By \cite[Remark 9.6.4]{LazII}  $\mu_*(\mu'_*(\mathcal{O}_Z(-E))= (\mu' \circ \mu)_*(\mathcal{O}_Z(-E)) = \overline{\mathfrak{a}}$, thus the thesis follows, since $\mu_*$ preserves inclusions of sheaves.
\end{proof}

\begin{lemma} \label{dasotto}
Let $\mu: Y \rightarrow X$ be a birational morphism between normal projective varieties. Let $B$ be an effective Weil divisor on $Y$. Then there exists an effective Cartier divisor $D$ on $X$ such that $\mu^*(D) \geq B$.
\end{lemma}
\begin{proof}
If we write $B$ as the sum of its exceptional and not exceptional part, it this then clear that we just need to show that for every exceptional prime divisor $E$ on $Y$ there exists an effective Cartier divisor $D$ on $X$ such that $\mu^*(D) \geq E$. 
By \cite[Theorem II.7.17]{Hartshorne} $\mu$ is given by the blow-up of a coherent sheaf of ideals $\mathcal{J}$ on $X$ and $\mathcal{J} \cdot \mathcal{O}_Y=\mathcal{O}_Y(-F)$ where $F$ is an effective Cartier divisor. Notice that $F$ contains in its support every $\mu$-exceptional divisor on $Y$ by \cite[1.42]{Debarre}, hence we are just reduced to find $D$ such that $\mu^*(D) \geq F$.\\
Take $D$ effective Cartier divisor such that $\mathcal{O}_X(-D) \subseteq \mathcal{J}$. 
Thus $\mathcal{O}_X(-D) \cdot \mathcal{O}_Y \subseteq\mathcal{J} \cdot \mathcal{O}_Y$ and this implies that $\mathcal{O}_Y(-\mu^*(D)) \subseteq \mathcal{O}_Y(-F)$, that is, $\mu^*(D) \geq F$.

\end{proof}

%
\section{Okounkov bodies and restricted volumes} \label{sect:okounkov}
%
\subsection{Okounkov bodies}  \label{subsect:okounkov}
Inspired by the fundamental paper by Lazarsfeld and Musta{\c{t}}{\u{a}} (\cite{LM}) we exploit the convex geometry of Okounkov bodies associated to linear series in order to study properties of the restricted volume of effective divisors.

We start by briefly recalling the notation associated to Okounkov bodies, referring to \cite{LM} or the original papers by Okounkov (\cite{O1}, \cite{O2}) for a more detailed  exposition. 

Let $V$ be a variety of dimension $d$. An \textit{admissible flag} is a chain of  subvarieties of $V$, $$F_\bullet: X =F_0 \supset F_1 \supset \dots \supset F_{d-1} \supset F_d=\{pt\},$$ such that each $F_i$ is irreducible, smooth at the point $F_d$ and $\codim_{X}(F_i)=i$.

Given an admissible flag $F_\bullet$ and a Cartier divisor $D$ we can define recursively a valuation-like function $$\nu=\nu_{F_\bullet}=\nu_{F_\bullet,D}: H^0(V, \mathcal{O}_V(D)) - \{0\}\rightarrow \mathbb{Z}^d,\ s\mapsto (\nu_1(s),\ldots, \nu_d(s)), $$
as follows:
\begin{enumerate}
\item[(1)] $D_0=V(s)$;
\item[(2)] $\nu_i:=\nu_i(s)$ is the coefficient of $F_i$ in $D_{i-1}$;
\item[(3)] $D_i=(D_{i-1}-\nu_iF_i)_{|{F_i}}$.
\end{enumerate}

When $V$ is projective then each $H^0(V,\mathcal{O}_V(mD))$ is a finite dimensional vector space, hence we can define the \textit{graded semigroup} of $D$ as the sub-semigroup $$\Gamma(D)=\Gamma_{F_\bullet}(D)=\{(\nu(s),m) | s \in H^0(V, \mathcal{O}_V(mD))-\{0\}, m \geq 0\} \subseteq \mathbb{N}^d \times \mathbb{N}=\mathbb{N}^{d+1}.$$
For each $m \geq 0$ we may also define $$\Gamma(D)_m=\Gamma_{F_{\bullet}}(D)_m= \mathrm{Im}\left( \left(H^0(V, \mathcal{O}_V(mD))-\{0\}\right) \stackrel{\nu}{\rightarrow}\mathbb{Z}^d \right).$$

In the above setting, let $\Sigma(D) \subseteq \mathbb{R}^{d+1}$ be the closed convex cone (with vertex at the origin) spanned by the semigroup $\Gamma(D)$. Then the \textit{Okounkov body} of $D$ relative to $F_{\bullet}$ is the compact convex set $$\Delta(D)=\Delta_{F_\bullet}(D)= \Sigma(D) \cap (\mathbb{R}^d \times \{1\}).$$ Equivalently: $$\Delta(D)= \textrm{closed convex hull }\left( \bigcup_{m \geq 1} \frac{1}{m} \cdot \Gamma(D)_m \right) \subseteq \mathbb{R}^d.$$

All the above constructions can be made in the same way also in the case of a (non-complete) graded linear series $W_\bullet$ belonging to a divisor $D$ on a variety $V$. Clearly in this case we need to consider a finite dimensional $W_m$ rather than $H^0(V,\mathcal{O}_V(mD))$ when defining the graded semigroup $\Gamma(W_{\bullet})$ and the Okounkov body $\Delta(W_{\bullet})$. In general $\Delta(W_\bullet)$ is a closed, convex set (see \cite[Remark 1.17]{LM}). If $V$ is projective then $\Delta(W_\bullet)$ is also compact and in this case - as we shall see - it is an actual convex body (i.e.: it has non-empty interior) if and only if $\vol(W_\bullet)>0$. 

Notice that the Okounkov body is homogeneous, i.e., $\Delta(W_{h,\bullet})=h \Delta (W_{\bullet})$, where the latter denotes the homothetic image of $\Delta(W_{\bullet})$ under scaling by the factor $h$. For big divisors this has been proved in \cite[Proposition 4.1(ii)]{LM}, in general it follows from the lemma below:
\begin{lemma} \label{lem:homokounkov}
Let $W_\bullet$ be a graded linear series on a projective variety $V$ of dimension $d$. Fix an admissible flag. Let $h$ be a positive integer. Then  $\Delta(W_{h,\bullet})= h \Delta(W_{\bullet})$.
\end{lemma}
\begin{proof}
By definition $\Delta(W_{h,\bullet})=\textrm{c.c.h.} \left( \bigcup_{m \geq 1} \frac{1}{m} \cdot \Gamma(W_{h,\bullet})_m \right)$, where c.c.h.\ means closed convex hull. Clearly $\textrm{c.c.h.} \left( \bigcup_{m \geq 1} \frac{1}{m} \cdot \Gamma(W_{h,\bullet})_m \right)= h \cdot \textrm{c.c.h.} \left( \bigcup_{m \geq 1} \frac{1}{mh} \cdot \Gamma(W_{\bullet})_{mh} \right)$, hence we just need to prove that $$ \bigcup_{m \geq 1} \frac{1}{mh} \cdot \Gamma(W_{\bullet})_{mh} = \bigcup_{m \geq 1} \frac{1}{m} \cdot \Gamma(W_{\bullet})_{m} .$$
The inclusion $\subseteq$ is obvious. Let us concentrate on the opposite inclusion $\supseteq$. Fix any $m$ and any vector $\vec{\sigma} $ in $\frac{1}{m} \Gamma(W_\bullet)_m$. By definition there exists $s \in W_m \setminus \{0\}$ such that $\vec{\sigma} = \frac{1}{m} (\nu_1(s), \ldots, \nu_d(s))$.   Consider $s^{\otimes h}$ (or better: its image) $\in W_{mh}$. By the valuation-like property of $\nu$ (\cite[\S 1.1,(iii)]{LM}) $(\nu_1(s^{\otimes h}), \ldots, \nu_d(s^{\otimes h}))=h (\nu_1(s), \ldots, \nu_d(s))$, therefore $\vec{\sigma} = \frac{1}{mh} \nu(s^{\otimes h}) \in \frac{1}{mh} \Gamma(W_\bullet)_{mh}$ and the statement follows. 
\end{proof}
The above lemma immediately implies that the volume of Okounkov bodies is a homogeneous function of degree $d$. 

\subsection{Asymptotic intersection number} \label{sect:asymptoticintersection}
Let $X$ be a projective variety and let $D$ be a Cartier divisor, or line bundle, on $X$. Let $V \subseteq X$ be a $d$-dimensional subvariety such that $V \not \subseteq \mathbb{B}(D)$, where $\mathbb{B}(D)$ is the stable base locus of $D$ (see \cite[Definition 2.1.20]{LazI}).
\\ For every $m\in \mathbb{M}(D)$ consider $u_m: X_m \rightarrow X$ a resolution of the base ideal $\mathfrak{b}(|mD|)$ and the related decomposition $u_m^*(|mD|)=|M_m|+E_m$, where $M_m$ is base point free and $E_m$ is a fixed effective divisor. Since $V$ is not contained in $\mathbb{B}(D)$, we can choose $u_m$ such that it is an isomorphism over the generic point of $V$. Call $V_m$ the strict transform of $V$ under $u_m$, then $\Supp(E_m)$ does not contain $V$.

The \textit{asymptotic intersection number} of $D$ and $V$ is just $$\|D^d\cdot V\|:=\limsup_{m \rightarrow + \infty} \frac{M_m^d \cdot V_m}{m^d},$$
where the $\limsup$ is taken over all $m \in \mathbb{M}(D)$. Actually the $\limsup$ is a limit (see \cite[Remark 2.9]{ELMNPrestricted}). Moreover, as noted in \cite[Remark 2.7]{ELMNPrestricted}, we can give a nice geometric interpretation of $\|D^d \cdot V\|$: since the number $M_m^d \cdot V_m$ is equal to the number of points in $D_1 \cap \dots \cap D_d \cap V$ outside $\mathrm{Bs}(|mD|)$, where $D_1, \ldots, D_d$ are general elements in $|mD|$, then 
$$\|D^d \cdot V\|= \limsup_{m \rightarrow + \infty} \frac{\sharp \left( V \cap D_{m,1} \cap \dots \cap D_{m,d} \setminus \mathrm{Bs}(|mD|)\right)}{m^d/d!}. $$

\subsection{Restricted volume and Fujita's Approximation Theorem}
The main graded linear series we are interested in are the ones that appear in the definition of the restricted volume. Given a smooth projective variety $X$, a $d$-dimensional subvariety $V \subseteq X$ and a divisor $D$ on $X$, we define $$H^0(X|V, \mathcal{O}_X(mD)):= \mathrm{Im}\left(H^0(X,\mathcal{O}_X(mD)) \stackrel{\mathrm{restr}_V}{\rightarrow} H^0(V, \mathcal{O}_V(mD))\right) \subseteq H^0(V, \mathcal{O}_V(mD)),$$ $$h^0(X|V,\mathcal{O}_X(mD)):=\dim\left(H^0(X|V, \mathcal{O}_X(mD))\right),$$ $$\vol_{X|V}(D):= \limsup_{m \rightarrow + \infty} \frac{h^0(X|V,\mathcal{O}_X(mD))}{m^d/d!}.$$ All the vector spaces $H^0(X|V, \mathcal{O}_X(mD))$ clearly form a graded linear series on $V$. Recall also that given any graded linear series $W_\bullet$ on a variety $V$ of dimension $d$, then we can analogously define its volume as $\vol(W_\bullet)=\vol_V(W_\bullet):= \limsup_{m \rightarrow +\infty} \frac{\dim(W_m)}{m^d/d!}$. 

First of all notice that the volume function is homogeneous. When $W_m=H^0(X|V, \mathcal{O}_X(mD))$ this is proved in \cite[Lemma 2.2]{ELMNPrestricted}.

\begin{lemma} \label{lem:homvolume}
Let $W_\bullet$ be a graded linear series on a projective variety $V$ of dimension $d$. The volume function is homogeneous of degree $d$, i.e., for every $h \in \mathbb{N}^+$, $\vol(W_{h, \bullet})=h^d\vol(W_\bullet)$.
\end{lemma}
\begin{proof}
Let $e=e(W_\bullet)$ be the exponent of $W_\bullet$. For $m$ sufficiently large we have that $W_m \not = 0 \Leftrightarrow e | m$. Hence $\limsup_{m \rightarrow +\infty} \frac{\dim(W_m)}{m^d/d!}=\limsup_{m \rightarrow +\infty} \frac{\dim(W_{em})}{{(em)}^d/d!}$. Let $c=\mathrm{lcm}(e,h)=e\cdot t$. Again, for $m$ sufficiently large, $W_{hm} \not = 0 \Leftrightarrow e | hm$ hence $\vol(W_{h,\bullet})/h^d=\limsup_{m \rightarrow +\infty} \frac{\dim(W_{hm})}{(hm)^d/d!} =\limsup_{m \rightarrow +\infty} \frac{\dim(W_{cm})}{(cm)^d/d!}$, 
where the latter is $\limsup_{m \rightarrow +\infty} \frac{\dim(W_{tem})}{(tem)^d/d!}$. At this point we just need to prove that $\limsup_{m \rightarrow +\infty} \frac{\dim(W_{em})}{{(em)}^d/d!}=\limsup_{m \rightarrow +\infty} \frac{\dim(W_{tem})}{(tem)^d/d!}$ and thus, for simplicity, we can suppose that $e=1$. The proof now copies that of \cite[Lemma 2.2.38]{LazI}.
\end{proof}

\begin{definition}\label{def:condGF}{\bf{(Condition (GF)).}}
We will say that a graded linear series $W_\bullet$ on a variety $V$ 
satisfies condition (GF) if $W_m\not=0$ for all $m\gg 0$, and for all sufficiently large $m$ the rational map
$$
 \phi_m : V\dashrightarrow \mathbb{P}(W_m)
$$
defined by $|W_m|$ is generically finite onto its image. 
 \hfill $\square$ \end{definition}
\begin{remark}
In the above definition, it is actually enough to check that $W_m\not=0$ for all $m\gg 0$ and that for a certain $m_0$ the rational map $\phi_{m_0}$ is generically finite onto its image. 
 \hfill $\square$ \end{remark}

\begin{remark}
If $f:X\to Y$ is (birationally equivalent to) the Iitaka fibration of $L\in \Pic(X)$ and $V\subset X$ is a generic complete intersection of dimension $\leq \kod(X,L)$, then the restricted linear series $W_\bullet$ on $V$ defined by $W_m:=H^0(X|V,mL)$ satisfies condition (GF), but in general the degree of the associated maps $\phi_m$ will be $>1$. For this reason it is interesting to extend to this more general case the results obtained by Lazarsfeld and Must\u a\c ta in \cite{LM} where the maps $\phi_m$ are assumed to be birational onto the image.
 \hfill $\square$ \end{remark}

Condition (GF) is a necessary assumption if one is interested in studying the volume of $W_\bullet$, as shown by the following lemma. Unlike in the complete case, it does not imply that  $\phi_m$ becomes birational onto the image, for large $m$.

\begin{lemma}\label{lem:necessary}
Let  $W_\bullet$ be a graded linear series  on a variety $V$ of dimension $d$ such that $\vol(W_\bullet)>0$. 
Then $W_{e, \bullet}$ satisfies condition (GF), where $e=e(W_\bullet)$ is the exponent of $W_\bullet$.
\end{lemma}
\begin{proof}
The hypothesis $\vol(W_\bullet)>0$ implies in particular that 
\begin{equation}\label{eq:infty}
 \limsup_{m\to +\infty} \frac{\dim(W_m)}{m^q}=+\infty,\ \forall q<d.
\end{equation}
Call $L$ the line bundle which $W_\bullet$ belongs to. Suppose that for all $m_0$ there exists $m>m_0$ such that 
$\phi_{W_m}$ is not generically finite. 
Let $Y\subset V$ be a general complete intersection of dimension $q<d$. 
Consider the short exact sequence of $Y$ in $V$ tensored with $mL$:
$$
0\to I_Y\otimes mL\to mL \to mL_{|Y}\to 0
$$
and the induced restriction map in cohomology
$$
0\to  H^0(V,I_Y\otimes mL)\to H^0(V,mL)\supseteq W_m\to H^0(Y,mL_{|Y}).
$$
By (\ref{eq:infty}) we have 
$$
 W_m\cap H^0(V,I_Y\otimes mL)\not= 0,
$$
i.e., there exists a non-zero section $s\in W_m$ such that $s_{|Y}\equiv 0.$
Now, the zeroes of $s$ must be in the direction of the fibers of $\phi_{W_m}$, plus some fixed divisor which is independent of $s$ (it depends on the base locus of $W_m$). Since $Y$ can be arbitrarily taken this leads to a contradiction.
\end{proof}

\begin{example} \label{ex:serieristretta}
Let $X$ be a smooth projective variety and $L\in \Pic(X)$ such that
$\kod(X,L)\geq 0$. Let $V\subseteq X$ be a subvariety such that $\vol_{X|V}(L)>0$. Let $e:=e(L)$ be the exponent of $L$. 
Notice that $\vol_{X|V}(L^{\otimes e})=e^{\dim(V)}\cdot \vol_{X|V}(L)>0$ by \cite[Lemma 2.2]{ELMNPrestricted}. 
Then, by Lemma \ref{lem:necessary}, the linear series $W_\bullet$ on $V$ given by 
$$
 W_m := H^0(X|V, L^{\otimes em})\subseteq H^0(V, L^{\otimes em}_{|V})
$$ 
satisfies condition (GF).
\hfill $\square$
\end{example}

The key result is the following.

\begin{lemma}\label{lem:condGF}
If a graded linear series $W_\bullet$ on a variety $V$ of dimension $d$
satisfies condition (GF) 
then there exists an admissible flag $F_\bullet$ on $V$ with respect to which the graded semigroup 
$$\Gamma:=\Gamma_{F_\bullet} (W_\bullet)\subseteq \mathbb N^{d+1}$$ 
generates $\mathbb Z^{d+1}$ as a group. 
\end{lemma}
\begin{proof}
 Take $\ell\in \N$ such that $|W_\ell|$ determines a map $\phi_\ell$ which 
 is generically finite onto its image. 
Let $y\in V$ be a smooth point at which  $\phi_\ell$ is defined and such that its differential has maximal rank. Suppose moreover that $y\in V$ is not contained in the base locus of $|W_q|$ for some fixed $q$ large enough and prime to $\ell$. Let
$$
 F_\bullet : V\supset F_1\supset F_2 \supset \ldots \supset F_{d-1}\supset F_d=\{y\} 
$$
be any admissible flag centered at $y$. 

For any prime number $p$ sufficiently large consider a resolution of 
the map $\phi_{p\ell}$, that is a proper birational morphism 
$$
\mu : V'\to V\ : \mu^*|W_{p\ell}|= |W'_{p\ell}| + E_{p\ell}
$$
where $ |W'_{p\ell}|$ is a base point free linear series on $V'$, with $ W_{p\ell} \cong W'_{p\ell}$, and $E_{p\ell}$  is a fixed effective divisor.  Since  $|W'_{p\ell}|$ determines a morphism which is generically finite onto its image, by the choice of $y$ we can find (by pulling back suitable hypersurfaces in $\mathbb{P} (W'_{p\ell})$) sections $t_0,\ldots, t_d\in W'_{p\ell}\cong W_{p\ell}$ such that
$$
 \nu_{F_\bullet}(t_0)=0,\ \ \nu_{F_\bullet}(t_i)=e_i,\ \forall i=1,\ldots, d,
$$
 where $e_i \in \mathbb Z^d$ is the $i$-th standard basis vector.  

On the other hand, since $y\not\in \Bs|W_q|$, there exists $s_0\in W_q$
such that $\nu_{F_\bullet}(s_0)=0$. 

Therefore we obtain that 
$$
(0,p\ell),(e_1,p\ell),\ldots,(e_d,p\ell)\in \Gamma \ {\textrm {and}}\ (0,q)\in \Gamma 
$$
and we are done.
 \end{proof}
 
\begin{remark}\label{rmk:countable}
Notice that given countably many graded linear series $W_{\bullet,\alpha}$ each satisfying condition (GF), there exists a flag $F_\bullet$ for which the conclusion of the previous lemma holds simultaneously for all of them. Indeed it is sufficient to take $y$ outside the countable union of Zariski closed subsets of $V$ where the associated rational maps ramify or are not defined. 
 \hfill $\square$ \end{remark} 
Following \cite{LM}, we will now state two theorems that relate the asymptotic volume of a graded linear series and the volume of its Okounkov body. We will work in the setting of complex projective variety but, using Lemma \ref{lem:condGF}, we will consider graded linear series that give generically finite (and not just birational) rational maps. See also \cite[Corollary 3.11]{KK} for a proof over an algebraically closed field.  We denote by $\vol_{\mathbb{R}^d}$ the standard Euclidean volume on $\mathbb{R}^d= \mathbb{Z}^d \otimes \mathbb{R}$, normalized so that the unit cube $[0,1]^d$ has volume $1$.
\begin{theorem} \label{thm:volume}
Let $V$ be a projective variety of dimension $d$ and let $W_{\bullet}$ be a graded linear series associated to a line bundle $L$. Suppose that $W_{\bullet}$ satisfies condition (GF). Then there exists an admissible flag $F_{\bullet}$ such that $$\vol_{\mathbb{R}^d}(\Delta(W_{\bullet}))=\frac{1}{d!} \vol(W_{\bullet}).$$
\end{theorem}
\begin{proof}
Let $F_{\bullet}$ be an admissible flag so that the conclusion of Lemma \ref{lem:condGF} holds. Hence $\Gamma_{F_{\bullet}}(W_{\bullet})=\Gamma(W_{\bullet})$ is a semigroup that verifies conditions (2.3), (2.5) in \cite{LM}. Also condition (2.4) in \cite{LM} holds, because $V$ is projective (see \cite[Definition 2.4]{LM}). Therefore, by \cite[Proposition 2.1]{LM}, $$\lim_{m \rightarrow + \infty} \frac{\sharp \Gamma_m}{m^d}= \vol_{\mathbb{R}^d}(\Delta(W_{\bullet})).$$ By \cite[Lemma 1.4]{LM} $\sharp \Gamma_m = \dim(W_m)$, and so the conclusion follows.
\end{proof}
 Now we can prove that also the opposite of Lemma \ref{lem:necessary} holds.
\begin{corollary} \label{cor:sufficient}
Let $W_\bullet$ be a graded linear series on a projective variety $V$ of dimension $d$. Then $W_{e,\bullet}$ satisfies condition (GF) if, and only if, $\vol(W_\bullet)>0$.
\end{corollary}
\begin{proof}
One implication is given by Lemma \ref{lem:necessary}.
Since $W_{e,\bullet}$ satisfies condition (GF) then, by Lemma \ref{lem:condGF}, there exists an admissible flag $F_\bullet$ such that $\Gamma:=\Gamma_{F_\bullet}(W_{e,\bullet})$ generates $\mathbb{Z}^{d+1}$. In particular $\Gamma$ contains $d+1$ linearly independent vectors $\in \mathbb{N}^{d+1} \subseteq \mathbb{R}^{d+1}$. Therefore by construction $\vol_{\mathbb{R}^d}(\Delta(W_{e,\bullet})) > 0$.  By Theorem \ref{thm:volume} this implies that $\vol(W_{e,\bullet})>0$ and hence $\vol(W_\bullet)>0$.
\end{proof}

\begin{remark}\label{rmk:big}
It makes sense to say that a graded linear series $W_\bullet$ is big if and only if $\vol(W_\bullet)>0$; when $V$ is projective, it follows from Lemma \ref{lem:necessary} and Corollary \ref{cor:sufficient} that this is equivalent to requiring the existence of an $m \in \mathbb{N}$ such that $W_m$ gives a generically finite map onto the image.
\hfill $\square$ 
\end{remark}
It is worth stressing that Theorem \ref{thm:volume} is valid without any assumption on $W_\bullet$, i.e., in the form of Theorem \ref{thm:volumeA}. 
\begin{proof}[Proof of Theorem \ref{thm:volumeA}]
If $\vol(W_\bullet)>0$ than we can apply Lemma \ref{lem:necessary} and conclude by invoking the homogeneity property both of the volume of Okounkov bodies (Lemma \ref{lem:homokounkov}) and of the volume of graded linear series (Lemma \ref{lem:homvolume}). If otherwise $\vol(W_\bullet)=0$, then, for every $m$, $W_{em}$ induces a rational map $\phi_{{em}}: V \dashrightarrow U_{em}$ that is never generically finite. 
Let $z=\max_{m \in \mathbb{N}}\{\dim(\phi_{{em}}(V))\} < d$. Taking a multiple $t$ of $e$ we can suppose that $\dim(\phi_{tm}(V))=z$ for all $m \in \mathbb{N}^+$.  Consider the maps $\phi_t$, $\phi_{tm}$ and the map $\psi_m$ given by $W_t^m$, the image of $W_t^{\otimes m}$ in $H^0(V,tmL)$ under the multiplication map. Note that $\psi_m$ is just $\phi_t$ followed by a Veronese embedding $v_m$. Let $\tilde{U}_t$ be the image of $U_t$ under $v_m$. Since, by definition of graded linear series, $W_t^m \subseteq W_{tm}$ , then there is a projection $p_m: U_{tm} \dashrightarrow \tilde{U}_t$ that makes the following diagram commute: 
\begin{displaymath}
\xymatrix{ V  \ar @{-->} [d] _{\phi_t} \ar @/_3pc/ @{-->} [dd]_{\psi_m}\ar @{-->} [rr]^{\phi_{tm}} & & U_{tm} \ar @{-->} [lldd]^{p_m} \\
U_t \ar[d]_{v_m} & & \\
\tilde{U}_t &  &\\
}
\end{displaymath}

Let $F_\bullet$ be an admissible flag on $V$ and such that $F_1$ dominates $U_{t}$ through the map $\phi_t$. This choice is feasible because $z < d$. 
Since $\dim(U_{tm})=\dim(\tilde{U}_t)$ and the diagram commutes, then $F_1$ dominates also $U_{tm}$ through $\phi_{tm}$, for every $m \in \mathbb{N}^+$. Therefore every section $s_{tm}$ of $W_{tm} \setminus \{0\}$ is not identically $0$ on $F_1$. This means that $\nu(s_{tm})=(0, \nu_2(s_{tm}), \ldots, \nu_d(s_{tm}))$, i.e., $\Gamma(W_\bullet)_{tm} \subseteq \mathbb{Z}^{d-1}$. Therefore, given $F_\bullet$ as above, $\Delta(W_{t,\bullet}) \subseteq \mathbb{R}^{d-1}$ and hence $\vol_{\mathbb{R}^{d}}(\Delta(W_{t,\bullet}))=0$. We can now conclude by Lemma \ref{lem:homokounkov}.
\end{proof}
 \begin{corollary}
Let $X$ be a  projective variety and let $L \in \Pic(X)$. 
Let $V$ be a subvariety of $X$ and set $d= \dim(V)$. Then $\vol_{X|V}(L)$ is actually a limit, that is: $$\vol_{X|V}(L)=\lim_{m \rightarrow +\infty} d! \frac{h^0(X|V, mL)}{m^d}.$$
 \end{corollary}
 \begin{proof}
By the homogeneity property of the restricted volume, 
we can suppose that $e(L)=1$. If $\vol_{X|V}(L)=0$ there is nothing to prove. Otherwise, by Example \ref{ex:serieristretta}, the graded linear series $W_m := H^0(X|V,mL)$ satisfies condition (GF). Hence by the proof of Theorem \ref{thm:volume} there exists an admissible flag $F_{\bullet}$ such that  $\lim_{m \rightarrow +\infty} d! \frac{h^0(X|V, mL)}{m^d}=d! \vol_{\mathbb{R}^d}(\Delta(W_{\bullet}))$. In particular such a limit exists and thus it must coincide with the $\limsup$.
 \end{proof}
 \begin{theorem} \label{thm:fujitaseries}
Let $V$ be a projective variety of dimension $d$  and let $W_{\bullet}$ be a graded linear series associated to $L \in \Pic(V)$. 
Suppose that $W_{\bullet}$ satisfies condition (GF). Set $T_{k,p}:=\mathrm{Im}(S^kW_p \rightarrow W_{pk})$. Then for every $\epsilon >0$ there exists $p_0(\epsilon)$ such that for every $p \geq p_0(\epsilon)$ we have that $$\lim_{k \rightarrow + \infty} \frac{\dim(T_{k,p})}{k^dp^d} \geq \frac{1}{d!}\vol(W_\bullet) - \epsilon.$$
 \end{theorem}
 \begin{proof}
Let $F_{\bullet}$ be an admissible flag so that the conclusion of Lemma \ref{lem:condGF} holds. Hence, as in Theorem \ref{thm:volume}, $\Gamma_{F_{\bullet}}(W_{\bullet})=\Gamma(W_{\bullet})$ is a semigroup that verifies conditions (2.3), (2.4), (2.5) of \cite{LM}. Then by \cite[Proposition 3.1]{LM}, for every $\epsilon >0 $ there exists $p_0(\epsilon)$ such that for all $p \geq p_0(\epsilon)$ $$\lim_{k \rightarrow + \infty} \frac{\sharp(k \ast \Gamma_p)}{k^dp^d} \geq \vol_{\mathbb{R}^d}(\Delta(W_{\bullet})) - \epsilon,$$ where $k \ast \Gamma_p:=\{\gamma_1 + \cdots + \gamma_k \mid \gamma_1, \ldots, \gamma_k \in \Gamma_p\}$.
By Theorem \ref{thm:volume}, $\vol_{\mathbb{R}^d}(\Delta(W_{\bullet})) = \frac{1}{d!} \vol(W_{\bullet})$. Moreover since $\nu=\nu_{F_{\bullet}}$ behaves like a valuation, then $k \ast \Gamma_p \subseteq\mathrm{Im}(T_{k,p}- \{0\} \stackrel{\nu}{\rightarrow} \mathbb{N}^d)$. We can now  conclude just noticing that, by \cite[Lemma 1.4]{LM}, $\sharp(\mathrm{Im}(T_{k,p}- \{0\} \stackrel{\nu}{\rightarrow} \mathbb{N}^d))=\dim(T_{k,p})$.
 \end{proof}
 
In \cite{ELMNPrestricted} interesting properties about the restricted volume of big divisors are proved, essentially by extending the celebrated Fujita's Approximation Theorem (see, for example, \cite[\S 11.4]{LazII}) from volume to restricted volume. In \cite[Remark 3.6]{LM}, the two authors remark that similar statements for big divisors can be also deduced from the convex geometry of Okounkov bodies.
Given Theorem \ref{thm:fujitaseries}, our purpose is now to derive Fujita-type results for restricted volumes also when $D$ is not necessarily big.
The first step is then to understand, given a semiample divisor $M$ on a projective variety $X$, how $\vol(M_{|V})$ and $\vol_{X|V}(M)$ are related, since the equality, $\vol(M_{|V})=\vol_{X|V}(M)$, for $V$ that contains a general point, holds only if $M$ in addition is also big (see \cite[Example 2.3]{ELMNPrestricted}):
  
  \begin{lemma} \label{lem:semiamplefib}
Let $X$ be a normal projective variety and let $M$ be a Cartier semiample divisor. Let $\phi: X \rightarrow Y$ be the semiample fibration associated to $M$. Suppose $V \subseteq X$ is a $d$-dimensional subvariety of $X$ such that $\phi_{|V}: V \rightarrow \phi(V)$ is generically finite of degree $\delta$. Then $\delta \vol_{X|V}(M)=M^d \cdot V=\vol(M_{|V})=\| M^d \cdot V \|$.
 \end{lemma}
 \begin{proof}
For the definition of semiample fibrations refer to \cite[\S 2.1.B]{LazI}.

Since all the formulas appearing in the statement are homogeneous of degree $d$ (see \cite[Lemma 2.2, Remark 2.10]{ELMNPrestricted}), then we can assume that $\overline{e}(M)=1$.  
By \cite[Theorem 2.1.27]{LazI} there exists an ample Cartier divisor $A$ on $Y$ such that $M=\phi^*(A)$. Since $\phi$ is an algebraic fiber space then, for every $k \geq 0$, $H^0(X,\phi^*(kA))\cong H^0(Y,kA)$ and, therefore, also $H^0(X,{I}_V(\phi^*(kA))) \cong H^0(Y,I_{\phi(V)}(kA))$, where $I_V, I_{\phi(V)}$ are the ideals defining $V$ and $\phi(V)$, respectively. 
Thus, for every $k \geq 0$, $h^0(X|V, kM)= h^0(Y|\phi(V), kA)$. Hence $\vol_{X|V}(M)= \vol_{Y|\phi(V)}(A)=A^d \cdot \phi(V)$, where the latter equality comes from Serre's vanishing (see \cite[Example 2.3]{ELMNPrestricted}). Since $\phi$ is generically finite of degree $\delta$ then, by projection formula, $M^d \cdot V= \phi^*(A)^d \cdot V = \delta A^d \cdot \phi(V)$ and the conclusion follows.
An analogous statement, for $V$ that contains a general point and $M$ just nef and abundant, can be found in \cite[Proposition 3.3]{PT}.
 \end{proof}

We are now ready to prove the main theorem of this section.

 \begin{proof}[Proof of Theorem \ref{thm:fujita}]
First of all notice that since $\vol_{X|V}(L)>0$ then in particular 
$V \not \subseteq \mathbb{B}(L)$ (i.e., $\|L^d \cdot V\|$ is well defined) and, by Lemma \ref{lem:necessary}, for $m$ sufficiently large and divisible ${\phi_{|mL|}}$ is generically finite on $V$, hence its degree is well defined. Moreover since both the restricted volume and the asymptotic intersection number are homogeneous, 
we can suppose that for every $m$, $\deg({\phi_{|mL|}}_{|V})=\delta$. 
Define the graded linear series  $W_m:=H^0(X|V,mL)$ and $T_{k,m}:=\mathrm{Im}(S^kH^0(X|V,mL) \rightarrow H^0(X|V,kmL))$. Given $\epsilon >0$, by Example \ref{ex:serieristretta}, we can apply Theorem \ref{thm:fujitaseries} and say that there exists $p_0(\epsilon)$ such that for every $p \geq p_0(\epsilon)$ \begin{equation} \label{eq:simmetrico}
\lim_{k \rightarrow + \infty} \frac{\dim(T_{k,p})}{k^dp^d} \geq \frac{1}{d!}\vol_{X|V}(L) - \epsilon.
\end{equation} 

For every $p \geq p_0(\epsilon)$, let $u_p: X_p \rightarrow X$ be the resolution of the map $\phi_p=\phi_{|pL|}: X \dashrightarrow Y_p$, so that $u_p^*(|pL|)=|M_p|+E_p$, where $|M_p|$ is base point free and $E_p$ is a fixed effective divisor. Since $V \not \subseteq \mathbb{B}(L)$,  for every $p \geq p_0(\epsilon)$ we can consider the strict transform $V_p$  of $V$ on $X_p$. Of course $V_p \not \subseteq \Supp(E_p)$. Since, for every $k$, $|kM_p| \subseteq |M_{pk}|$ then if we consider the semiample fibration $\psi_{(p)}: X_p \rightarrow Y_{(p)}$ associated to $M_p$ we have that $\deg({\psi_{(p)}}_{|V_p}:V_p \rightarrow \psi_{(p)}(V_p))=\delta$. Moreover, since $u_p$ is birational, we have that for every $q$, $H^0(X|V, qL) \cong H^0(X_p|V_p, u_p^*(qL))$. 
Therefore for every $k$, $\mathrm{Im}(S^kH^0(X|V,pL) \rightarrow H^0(X|V,kpL))$ is naturally isomorphic to $\mathrm{Im}(S^kH^0(X_p|V_p, u_p^*(pL)) \rightarrow H^0(X_p|V_p,u_p^*(kpL))) \subseteq H^0(X_p|V_p, kM_p)$. 

By (\ref{eq:simmetrico}) then we have that for every $p \geq p_0(\epsilon)$, $$\limsup_{k \rightarrow +\infty} \frac{h^0(X_p|V_p, kM_p)}{k^dp^d} \geq \frac{1}{d!}{\vol_{X|V}(L)}-\epsilon.$$
 
By Lemma \ref{lem:semiamplefib}, $\limsup_{k \rightarrow +\infty} \frac{h^0(X_p|V_p, kM_p)}{k^dp^d}=\frac{1}{d!}\frac{1}{p^d} \vol_{X_p|V_p}(M_p)= \frac{1}{\delta} \frac{1}{d!} \frac{1}{p^d} M_p^d \cdot V_p$, that is: $$\left(\frac{M_p}{p}\right)^d \cdot V_p
\geq \delta \vol_{X|V}(L)-d! \delta \epsilon.$$
On the contrary, consider that by \cite[Lemma 2.4]{ELMNPrestricted} $\vol_{X|V}(pL)=\vol_{X_p|V_p}(u_p^*(pL)) \geq \vol_{X_p|V_p}(M_p)$, then, by the homogeneity of the restricted volume, $\vol_{X|V}(L)\geq \frac{1}{p^d}\vol_{X_p|V_p}(M_p)$. 
By Lemma \ref{lem:semiamplefib}, the former is equal to $\frac{1}{\delta} \left( \frac{M_p}{p}\right)^d \cdot V_p$.
Set  $\mu_\epsilon:=u_{p_0(\epsilon)}$, $V_\epsilon:=V_{p_0(\epsilon)}$, $A_\epsilon:=\frac{M_{p_0(\epsilon)}}{p_0(\epsilon)}$, $E_\epsilon:=\frac{E_{p_0(\epsilon)}}{p_0(\epsilon)}$. (\ref{tesi1}), (\ref{tesi2}) and (\ref{tesi3}) are then satisfied. 
Moreover, by definition, $\|L^d \cdot V\|=\limsup_{p \rightarrow + \infty} \left(\frac{M_p}{p}\right)^d \cdot V_p$, so that, for what we have just said, $\|L^d \cdot V\|=\delta \vol_{X|V}(L)$.
 \end{proof}
\begin{remark} \label{rmk:delta}
If $V$ contains a general point, then $\delta=\deg(f_{|V}:V \dashrightarrow f(V))$, where $f$ is the Iitaka fibration related to $L$.
See also \cite[Corollary 1.2]{PT} and \cite[Notations and conventions, p. 3]{PT}.
 \hfill $\square$ \end{remark}
 
 \begin{remark} \label{rmk:zero}
The final part of the above proof implies that if $V \not \subseteq \mathbb{B}(L)$ but $\vol_{X|V}(L)=0$ then also $\|L^d \cdot V\|=0$. 
  \hfill $\square$ \end{remark}

\begin{proof}[Proof of Corollary \ref{cor:log-concavity}]
We can assume that both $\vol_{X|V}(L_1)$, $\vol_{X|V}(L_2)$ are positive. Then it goes exactly as in \cite[Theorem 11.4.9]{LazII}, once one replaces the Fujita approximation theorem \cite[Theorem 11.4.4]{LazII} with the Fujita approximation as in Theorem \ref{thm:fujita}. We just need to point out that if $A_1$, $A_2$ are two semiample $\mathbb{Q}$-divisors, if $\phi_1$, $\phi_2$ are their respective semiample fibrations and if we set $\delta_i = \deg ({\phi_i} _{|V}: V \rightarrow \phi_i(V))$ ($i=1,2$), then $A_1+A_2$  is a semiample divisor and its associated semiample fibration has degree on $V$ less than or equal to $\min\{\delta_1, \delta_2\}$.
\end{proof}
\begin{proof}[Proof of Corollary \ref{cor:caratt}]
Assume $d\leq\kappa(X,L) < \dim X$.
Then we can find a $d$-dimensional irreducible subvariety $V \subset X$, 
with $\delta=\deg (f|_V : V \dasharrow f(V)) > 1$
(it is sufficient to take  a general complete intersection over 
a general $d$-dimensional subvariety of $Y$).
By Theorem \ref{thm:fujita} we have $\delta\cdot \vol_{X|V}(L)=|| L^d\cdot V||.$
By \cite[Lemma 3.5]{PT}, we have that $\mu(V,L)\geq ||L^d\cdot V||$. 
Since by hypothesis $\mu(V,L)=\vol_{X|V}(L)$ we get a contradiction.
\end{proof}

%
\subsection{Multi-graded linear series} \label{ss:multigraded}
%

We will extend the results obtained in the previous paragraph to the more
general setting of multi-graded linear series. For the reader's convenience we recall several definitions from \cite[\S 4]{LM}.

The set-up is the following. Let $V$ be a variety of dimension $d$ and let $D_1,\ldots, D_r$ be Cartier divisors on $V$. 
For any 
$\vec{m}=(m_1,\ldots,m_r)\in \N^r$ we write 
$$
 \vec m D:=m_1D_1+\ldots +m_rD_r
$$
and we put $|\vec m|:= \sum_{i=1}^r |m_i|$.
\begin{definition}
A multi-graded linear series $W_{\vec \bullet}$ on $V$ associated to the divisors $D_1,\ldots,D_r$ is the data of finite-dimensional vector spaces
$$
 W_{\vec k}\subseteq H^0(V,\vec k D)
$$
for every $\vec k\in \N^r$, with $W_{\vec 0}=\C$, such that
$$
 W_{\vec k}\cdot W_{\vec m}\subseteq W_{\vec k+\vec m}. 
$$
 \hfill $\square$ \end{definition}
Notice that the multiplication $W_{\vec k}\cdot W_{\vec m}$ means the image of $W_{\vec k}\otimes W_{\vec m}$ under the natural map 
$H^0(V,\vec k D)\otimes H^0(V,\vec m D)\to H^0(V,(\vec k+\vec m) D)$.

Given a multi-graded linear series $W_{\vec\bullet}$ and  integral vector $\vec a\in \N^r$ we define a single graded linear series $W_{\vec a, \bullet}$ 
as follows:
$$
W_{\vec a, k}:=W_{k\vec a},\ \forall k\in\N.
$$

We will be interested in studying the volume 
$
 \vol (W_{\vec a, \bullet})
$
(assuming the quantity is finite) and its variation with $\vec a$. 

To this end we fix an admissible flag $F_\bullet$ on $V$ and set 
$$
 \Delta (\vec a):= \Delta (W_{\vec a, \bullet}).
$$
Finally, the support 
$$
 \supp(W_{\vec\bullet})\subset \R^r
$$
of $W_{\vec \bullet}$ is the closed convex cone spanned by all indices $\vec m\in \N^r$ such that $W_{\vec m}\not = 0$.

The crucial definition is the following.
\begin{definition}\label{def:GF'}
Let $V$ and $W_{\vec\bullet}$ as above.
We say that $W_{\vec\bullet}$ satisfies condition (GF') if the following hold:
\begin{enumerate}
\item[(i)] $\supp(W_{\vec\bullet})\subset \R^r$ has non-empty interior;
\item[(ii)] for any integral vector $\vec a\in {\textrm{int}}(\supp(W_{\vec \bullet}))$, we have $W_{k\vec a}\not = 0,$ for all $k\gg 0$;
\item[(iii)] there exists an integral vector $\vec a_0\in {\textrm{int}}(\supp(W_{\vec \bullet}))$ such that the single graded linear series 
$W_{\vec a_0, \bullet}$ satisfies condition (GF).
\end{enumerate}
 \hfill $\square$ \end{definition}
It is easy to check (cf. the following lemma) that item  (iii) above insures that any integral vector in the interior of the support of $W_{\vec \bullet}$ has the same property.

\begin{lemma} Let $V$ and $W_{\vec\bullet}$ as above.
Suppose $W_{\vec\bullet}$ satisfies condition (GF'). Then for every integral vector $\vec a\in {\mathrm{int}}(\supp(W_{\vec \bullet}))$ the corresponding single graded linear series 
$ W_{\vec a, \bullet}$ satisfies condition (GF).
\end{lemma}

\begin{proof}Let $\vec a\in {\textrm{int}}(\supp(W_{\vec \bullet}))$ be any integral vector. Then, for $k\in \N$ large enough, we have
$$
k\vec a = \vec a_0 + \vec b,
$$
where $\vec a_0$ is as in Definition \ref{def:GF'}, item (iii), and $\vec b$ also lies in ${\textrm{int}}(\supp(W_{\vec \bullet}))$. By Definition \ref{def:GF'}, item (ii), we have $W_{m\vec b} \not = 0$, for $m\gg 0$. Let $s\in W_{m\vec b}$ be a non-zero section. Then 
$$
 s\cdot W_{m\vec a_0}\subseteq W_{mk\vec a}
$$
and therefore since $W_{\vec a_0, \bullet}$ satisfies condition (GF) so does $W_{\vec a, \bullet}$.
\end{proof}
The lemma implies in particular that if a multi-graded linear series $W_{\vec \bullet}$ satisfies condition (GF'), then every  single graded linear series $W_{\vec a, \bullet}$  satisfies condition (GF) for every $\vec a\in {\mathrm{int}}(\supp(W_{\vec \bullet}))$. We want to realize the Okounkov bodies of the linear series $W_{\vec a, \bullet}$  as fibers of a global cone associated to $W_{\vec\bullet}$.
To do so, fix an admissible flag $F_\bullet$ on $V$ and define the multi-graded semigroup $\Gamma(W_{\vec\bullet})$ of $W_{\vec\bullet}$
as follows
$$
 \Gamma(W_{\vec\bullet}):=\Gamma_{F_\bullet}(W_{\vec\bullet})=
 \{(\nu(s),\vec m): 0\not = s\in W_{\vec m}\}\subseteq \N^{d+r}.
$$

\begin{lemma}\label{lem:2} Let $V$ and $W_{\vec\bullet}$ as above.
Suppose $W_{\vec\bullet}$ satisfies condition (GF'). Then there exists a flag $F_\bullet$ for which $\Gamma_{F_\bullet}(W_{\vec \bullet})$ generates $\Z^{d+r}$ as a group. 
\end{lemma}
\begin{proof}
 Given an integral vector $\vec a\in {\textrm{int}(\supp(W_{\vec\bullet}))}$
denote by 
$$
 \Gamma_{\vec a}=\Gamma_{F_\bullet}(W_{\vec a,\bullet})\subseteq 
 \N^d\times \N\vec a\subset \N^d\times \N^r
$$
the graded semigroup of $W_{\vec a,\bullet}$ with respect to $F_\bullet$, which is a sub-semigroup of $\Gamma(W_{\vec \bullet})$. By Remark \ref{rmk:countable} we can suppose that each $\Gamma_{\vec a}$ generates $\Z^d\times \N\vec a$
as a group. To conclude choose $\vec a_1,\ldots,\vec a_r$ spanning $\Z^r$ so that the corresponding $\Gamma_{\vec a_i}$ together generate $\Z^{d+r}$. Since $\Gamma_{\vec a_i}\subseteq \Gamma(W_{\vec \bullet})$
we are done.
\end{proof}
Now let $\Sigma(W_{\vec \bullet})\subseteq \R^d\times \R^r$ be the closed convex cone spanned by $\Gamma(W_{\vec \bullet})$ and set 
\begin{equation}\label{eq:globalOk}
 \Delta(W_{\vec \bullet}):=\Sigma(W_{\vec \bullet})\subseteq \R^d\times \R^r.
\end{equation}
Consider the projection $(pr_2)_{|\Delta(W_{\vec \bullet})}:\Delta(W_{\vec \bullet})\to \R^r$
induced by the restriction to $\Delta(W_{\vec \bullet})$ of the 
natural projection onto the second factor $\R^d\times \R^r\to  \R^r$. 


\begin{proof}[Proof of Theorem \ref{thm:infamiglia}]
Let $\Delta(W_{\vec \bullet})$ be defined as in (\ref{eq:globalOk}). By homogeneity, we can suppose that $\vec a$ is integral. Remark that by definition, if $\Gamma:=\Gamma(W_{\vec \bullet})$ and $\vec a\in {\textrm{int}}(\supp(W_{\vec \bullet}))$, then
$$
 \Gamma_{\N\vec a}:=\Gamma\cap(\N^d\times \N\vec a)= \Gamma(W_{\vec a,\bullet}).
$$
Hence 
$$\Delta(W_{\vec a, \bullet}):=\Delta (\Gamma(W_{\vec a,\bullet}))=\Sigma (\Gamma_{\N\vec a})\cap (\R^d\times \{\vec a\}).
$$
Since, by Lemma \ref{lem:2}, $\Gamma$ generates $\Z^{d+r}$, then by \cite[Proposition 4.9]{LM} we have that
$$
 \Delta (\Gamma_{\N\vec a})=\Sigma(\Gamma)\cap(\R^d\times \{\vec a\}) 
$$
and we are done.
\end{proof}

From the theorem above, using standard results from convex geometry, one immediately deduces interesting properties of the volume function. In particular Corollary \ref{cor:cont} follows. 
\begin{remark}\label{rmk:cont}
A special case of Corollary \ref{cor:cont} of particular interest is the following. Let $D_1,\ldots,D_r$ be effective divisors on a variety $X$ and let 
$$
 k:= {\textrm{max}} \{\kod(X,D_1),\ldots, \kod(X,D_r)\}.
$$
Let $V$ be a general subvariety of dimension $d\leq k$. 
Let $W_{\vec\bullet}$ be the multi-graded linear series on $V$ defined as follows:
$$
 W_{\vec m}:=H^0(X|V, \vec m D),\ \forall \vec m \in \N^r.
$$
Then Corollary \ref{cor:cont} implies that  the restricted volume function
$$
 \vol_{X|V} : \Q_{>0}^r\to \R,\ \vec m\mapsto \vol_{X|V}(\vec m D)
$$
extends to  a continuous function which is  log-concave. Since the restricted volume is a homogeneous function, then the same is true for non-effective divisors of non-negative Iitaka dimension.  \phantom{a} \hfill$\square$
\end{remark}

%
\subsection{Multi-graded Fujita approximation}
%

Given a multi-graded linear series $W_{\vec \bullet}$ and a positive integer $p$, we will consider   the sub-multi-graded linear series $W_{\vec \bullet}^{(p)}$ defined as follows:
$$
 W_{\vec m}^{(p)}:= \sum_{|\vec m_i|=p,\ \vec m_1+\ldots + \vec m_k=\vec m}  W_{\vec m_1}\cdots W_{\vec m_k}\subseteq W_{\vec m}, \ \textrm{if}\ |\vec m|=kp
$$ 
and 
$$
W_{\vec m}^{(p)}:=0\ \text{if}\ p\ \text{does not divide}\ |\vec m|.
$$
As usual, fixing an integral vector $\vec a\in \N^r$, the multi-graded  linear series $W_{\vec \bullet}^{(p)}$ gives rise to a single graded linear series 
$W_{\vec a, \bullet}^{(p)}$ defined as follows:
$$
 W_{\vec a, h}^{(p)}=\sum_{|\vec m_i|=p,\ \vec m_1+\ldots + \vec m_k=h\vec a}  W_{\vec m_1}\cdots W_{\vec m_k}\subseteq W_{h\vec a}, \ \textrm{if}\ |\vec a|\cdot h =kp
$$
and 
$$
W_{\vec a,h}^{(p)}:=0\ \text{if}\ p\ \text{does not divide}\ |\vec a|\cdot h .
$$
\begin{remark}\label{rmk:attenzione}
Since the notation is  particularly cumbersome, it is important to notice that when $p=p' \cdot |\vec a|$ the graded linear series $W_{\vec a, \bullet}^{(p)}$ defined above contains the graded linear series $(W_{\vec a, \bullet})^{(p')}$ associated to the single graded linear series $W_{\vec a, \bullet}$. In fact
 
$$
 W_{\vec a, h}^{(p)}=\sum_{|\vec m_i|=p=p'|\vec a|,\ \vec m_1+\ldots + \vec m_k=h\vec a}  W_{\vec m_1}\cdots W_{\vec m_k}\subseteq W_{h\vec a}, \ \textrm{if}\ |\vec a|\cdot h =kp=kp'|\vec a|\ (\Leftrightarrow h=kp'),
$$
and zero otherwise. 
On the other hand, by definition, the $h$-th graded piece of $(W_{\vec a, \bullet})^{(p')}$ is 
$$
 \Im(S^kW_{p'\vec a}\to W_{kp'\vec a}) \ \textrm{if}\ h=kp'
$$
and zero otherwise. So our claim is proved. 
\hfill$\square$
\end{remark}
\begin{theorem}\label{thm:multigraded}
Let $V$ be a projective variety of dimension $d$. Let $W_{\vec\bullet}$ be a multi-graded linear series on $V$ associated to $r$ divisors. Suppose $W_{\vec\bullet}$ satisfies condition (GF'). 
Then for any $\epsilon>0$ and any compact set 
$$
 K\subseteq \{ (a_1,\ldots,a_r)\in \R_{>0}^r:a_1+\ldots+a_r=1\} \cap \mathrm{int}(\supp(W_{\vec \bullet}))
$$
there exists an integer $p_0=p_0(\epsilon,K)$ such that if $p\geq p_0$ then
\begin{equation}\label{eq:multifujita}
 \left|1-\frac{ \vol({W_{\vec a,\bullet}^{(p)}})}{\vol(W_{\vec a,\bullet})} \right| < \epsilon
\end{equation}
for all integral vectors $\vec a$ in the cone $C_K\subset \R_{\geq 0}^r$ generated by $K$.

\end{theorem}
When the divisors $D_1\ldots,D_r$ are big the theorem was obtained by Jow \cite[see Theorem p. 333 and Remark p. 335]{Jow}. 

\begin{remark}\label{rmk:stessop}
The single-graded Fujita approximation proved in Theorem \ref{thm:fujitaseries} yields, for all integral vectors $\vec a$ and all $\epsilon>0$, the existence of an integer $p_0=p_0(\vec a,\epsilon)$ such that (\ref{eq:multifujita})  holds. The main point in Theorem \ref{thm:multigraded} is that $p_0$ is the same {\it for all} integral vectors $\vec a\in C_K$.\hfill$\square$
\end{remark}


\begin{proof}[Proof of Theorem \ref{thm:multigraded}]
We follow Jow's arguments, replacing the single graded Fujita approximation proved in the big case in \cite[Theorem 3.3]{LM} by our Theorem \ref{thm:fujitaseries} and \cite[Theorem 4.21]{LM} by our Theorem \ref{thm:infamiglia}. Since, by the homogeneity of the volume, (\ref{eq:multifujita}) is invariant under scaling of $\vec a$ it is sufficient to prove it for $\vec a\in K\cap \Q^r$. We consider the global Okounkov bodies $\Delta(W_{\vec \bullet})\subset \R^d\times \R^r$ (respectively $\Delta(W_{\vec \bullet}^{(p)})\subset \R^d\times \R^r$) associated to $W_{\vec \bullet}$ (respectively to $W_{\vec \bullet}^{(p)}$) as defined in (\ref{eq:globalOk}). By Theorem  \ref{thm:infamiglia} we have 
\begin{equation}\label{eq:fibre1}
 (pr_2)^{-1}_{|\Delta(W_{\vec \bullet})}(\vec a)=\Delta(W_{\vec a,\bullet})
\end{equation}
as well as 
\begin{equation}\label{eq:fibre2}
 (pr_2)^{-1}_{|\Delta(W^{(p)}_{\vec \bullet})}(\vec a)=\Delta(W^{(p)}_{\vec a,\bullet})
 \end{equation}
for all $\vec a\in K\cap \Q^r$. By the single graded Fujita approximation 
provided in Theorem \ref{thm:fujitaseries}, $\vol(W_{\vec a,\bullet})$ can be approximated arbitrarily well by $\vol((W_{\vec a,\bullet})^{(p')}),$ for $p'\gg0$, and the latter volume, by Remark \ref{rmk:attenzione}, is bounded from above by 
$\vol(W_{\vec a, \bullet}^{(p)})$. Hence, given any finite subset $S\subset K\cap \Q^r$ and any $\epsilon'>0$  we have 
$$
\vol(\Delta(W_{\vec a, \bullet}^{(p)}))\geq \vol(\Delta(W_{\vec a,\bullet})) -\epsilon'
$$
for all $\vec a\in S$ and $p\in\N$ sufficiently large and divisible. 
By (\ref{eq:fibre1}) and (\ref{eq:fibre2}) this yields 
\begin{equation}\label{eq:maggior}
\vol((pr_2)^{-1}_{|\Delta(W^{(p)}_{\vec \bullet})}(\vec a))\geq \vol((pr_2)^{-1}_{|\Delta(W_{\vec \bullet})}(\vec a)) -\epsilon'
\end{equation}
for all $\vec a\in S$ and $p\in\N$ sufficiently large and divisible.

The function $\vec a\mapsto \vol((pr_2)^{-1}_{|\Delta(W_{\vec \bullet})}(\vec a))$ is uniformly continous on $K$. Therefore for any $\epsilon'>0$ we can find a partition of $K$ into a union of polytopes with disjoint interiors $K=\cup_i K_i$ such that all the vertices of each $K_i$ have rational coordinates and for each $K_i$ there is a constant $M_i$ such that 
\begin{equation}\label{eq:Mi}
M_i\leq \vol((pr_2)^{-1}_{|\Delta(W_{\vec \bullet})}(\vec a))\leq M_i +\epsilon',\ \forall \vec a\in K_i.
\end{equation}
Now take $S$ to be the set of all the vertices of all the $T_i$'s. Then (\ref{eq:maggior}) holds. We claim that this implies  
\begin{equation}\label{eq:fine}
\vol ((pr_2)^{-1}_{|\Delta(W^{(p)}_{\vec \bullet})}(\vec a))\geq \vol((pr_2)^{-1}_{|\Delta(W_{\vec \bullet})}(\vec a)) -2\epsilon',\ \forall \vec a\in K\cap \Q^r.
\end{equation}
Before proving (\ref{eq:fine}) let us see how it implies the theorem. First of all, it is equivalent to 
$$
 1-\frac{\vol((pr_2)^{-1}_{|\Delta(W_{\vec \bullet})}(\vec a))}{\vol ((pr_2)^{-1}_{|\Delta(W^{(p)}_{\vec \bullet})}(\vec a))}\leq \frac{2\epsilon'}{\vol ((pr_2)^{-1}_{|\Delta(W^{(p)}_{\vec \bullet})}(\vec a))},\ \forall \vec a\in K\cap \Q^r.
$$
Since $\vec a\mapsto \vol ((pr_2)^{-1}_{|\Delta(W^{(p)}_{\vec \bullet})}(\vec a))$ is a continuous map and $K\cap \Q^r$ is compact, it is sufficient to take 
$$
\epsilon'\leq   \frac{ \epsilon}{2}\cdot \min_{\vec a\in K\cap \Q^r}(\vol ((pr_2)^{-1}_{|\Delta(W^{(p)}_{\vec \bullet})}(\vec a))) 
$$
and to invoke again (\ref{eq:fibre1}) and (\ref{eq:fibre2}) to conclude.

Now let us prove (\ref{eq:fine}). It is sufficient to prove it on each $K_i$. Let 
$\vec v_1,\ldots,\vec v_k$ be the vertices of $K_i$. Since $K_i$ is a polytope, each vector $\vec a\in K_i$ can be written
as $\vec a=\sum_j t_j\vec v_j$, where $t_j\geq 0$ and $\sum_j t_j=1$. 
Since $\Delta(W^{(p)}_{\vec \bullet})$ is convex we have 
$$ 
(pr_2)^{-1}_{|\Delta(W^{(p)}_{\vec \bullet})}(\vec a)\supseteq \sum_j t_j (pr_2)^{-1}_{|\Delta(W^{(p)}_{\vec \bullet})}(\vec v_j).
$$
By (\ref{eq:maggior}) and (\ref{eq:Mi}) the volume of the right-hand side is $\geq M_i-\epsilon'$. Therefore, by the Brunn--Minkowski inequality we get
\begin{equation}\label{eq:penultimo}
\vol((pr_2)^{-1}_{|\Delta(W^{(p)}_{\vec \bullet})}(\vec a))\geq M_i-\epsilon',\ \forall \vec a\in K_i\cap \Q^r.
\end{equation}
Equation (\ref{eq:fine}) then follows from (\ref{eq:penultimo}) and (\ref{eq:Mi}).
\end{proof}


\section{Multiplier and base ideals} \label{sect:multiplier}


\subsection{Properties \textnormal{($\star$)} and \textnormal{b-($\star$)}} 
In \cite[Theorem 11.2.21]{LazII} it is proved that if $B$ is a big divisor on a smooth projective variety $Z$, then there exists an effective divisor $N$ on $Z$ such that 
\begin{equation} \label{fondamentale}
\mathcal{J}(Z, \|pB\|)(-N) \subseteq \mathfrak{b}(|pB|) \tag{$\star$},
\end{equation}
for every $p \in \mathbb{N}(B)$.
The purpose of this section is to understand if (\ref{fondamentale}) holds for other kinds of divisors. 
Actually we are not really interested in (\ref{fondamentale}) but mostly on its birational counterpart. Therefore we give the following
\begin{definition} \label{pinco}
Given a normal projective variety $X$ and a divisor $D$ on $X$, we will say that $D$ satisfies property b-$(\star)$ if 
\begin{enumerate}
\item $D$ is Cartier, 
\item $\kappa(X,D) \geq 0$,
\item there exist a smooth projective variety $Y$ and a  birational map  $\mu: Y \rightarrow X$, such that $\mu^*(D)$ verifies (\ref{fondamentale}) (that is: there exists an effective divisor $N$ on $Y$ such that $\mathcal{J}(Y, \|\mu^*(pD)\|)(-N) \subseteq \mathfrak{b}(|\mu^*(pD)|)$ for every $p \in \mathbb{N}(D)$).
\end{enumerate}
 \hfill $\square$ \end{definition}
Note that, possibly taking a greater divisor $N' \geq N$, we can always assume that $N$ is ample.

We will collect here a series of lemmas that will be useful later on and that give some general properties of divisors verifying b-$(\star)$.

\begin{lemma} \label{quasibig}
Let $Y,Z$ be two smooth projective varieties and let $f: Y \rightarrow Z$ be a surjective morphism with connected fibers. Let $B$ be a divisor on $Z$ and let $\mathbb{P}$ be a subset of $\mathbb{N}(B)$. Suppose, for a certain effective divisor $N$ on $Z$, that $B$ satisfies equation \textnormal{(\ref{fondamentale})} for every $p \in \mathbb{P}$. Then $$ \mathcal{J}(Y, \|f^*(pB)\|)(-f^*(N)) \subseteq \mathfrak{b}(|f^*(pB)|)$$ for every $p \in \mathbb{P}$.
\end{lemma}

\begin{proof}
By equation (\ref{fondamentale}), taking preimages under $f$, we have that $(\mathcal{J}(Z, \|pB\|)(-N))\cdot \mathcal{O}_Y \subseteq \mathfrak{b}(|pB|) \cdot \mathcal{O}_Y$. Clearly $(\mathcal{J}(Z, \|pB\|)(-N))\cdot \mathcal{O}_Y = (\mathcal{J}(Z, \|pB\|) \cdot \mathcal{O}_Y) (\mathcal{O}_Z(-N) \cdot \mathcal{O}_Y)$ and, since $\mathcal{O}_Z(-N)$ is an invertible sheaf, $(\mathcal{O}_Z(-N) \cdot \mathcal{O}_Y)= \mathcal{O}_Y(-f^*(N))$ (see \cite[Caution II.7.12.2]{Hartshorne}). 
Moreover, given the fact that $f$ is an algebraic fiber space, $f^*|pB|=|f^*(pB)|$ and therefore $\mathfrak{b}(|pB|) \cdot \mathcal{O}_Y= \mathfrak{b}(|f^*(pB)|)$.
By \cite[Example 9.5.8]{LazII} we have that $\mathcal{J}(Y, \|f^*(pB)\|) \subseteq \mathcal{J}(Z, \|pB\|) \cdot \mathcal{O}_Y$. Actually Example 9.5.8 applies to multiplier ideals, but since $f^*|pB|=|f^*(pB)|$ the result extends to asymptotic multiplier ideals as well.
Putting everything together leads to the statement.
\end{proof}

In particular we see that if $D$ satisfies (\ref{fondamentale}) on $Y$ then the preimage of $D$ satisfies (\ref{fondamentale}) on any smooth projective variety that birationally dominates $Y$. Hence the following holds:

\begin{corollary} \label{prepinco}
Let $X', X$ be normal projective varieties and let $f: X' \rightarrow X$ be a surjective morphism with connected fibers. If $D$ verifies \textnormal{b-($\star$)} on $X$ then $f^*(D)$ verifies \textnormal{b-($\star$)} on $X'$. Hence \textnormal{b-($\star$)} is a birational property.
\end{corollary}

The following two easy lemmas explain that to prove that $D$ verifies b-$(\star)$, we just need to check (\ref{fondamentale}) for a sub-semigroup of $\mathbb{N}(D)$:

\begin{lemma} \label{numerofinito}
Let $Y$ be a smooth projective variety and let $D$ be a divisor such that $\kappa(Y,D) \geq 0$. Let $\mathbb{P}$ be a finite subset of $\mathbb{N}(D)$. Then there exists an effective divisor $N_\mathbb{P}$ such that $$\mathcal{J}(Y, \|pD\|)(-N_\mathbb{P}) \subseteq \mathfrak{b}(|pD|),$$ for every $p \in \mathbb{P}$.
\end{lemma}
\begin{proof}
For every $p \in \mathbb{N}(D)$ choose $D_p \in |pD|$. Clearly $\mathcal{O}_Y(-D_p) \subseteq \mathfrak{b}(|pD|)$. At this point just take $N_{\mathbb{P}}:= \sum_{p \in \mathbb{P}} D_p$. Since for every $p \in \mathbb{P}$, $\mathcal{O}_Y(-N_{\mathbb{P}}) \subseteq \mathfrak{b}(|pD|)$ then we have also that $\mathcal{J}(Y,\|pD\|)\mathcal{O}_Y(-N_{\mathbb{P}}) \subseteq \mathfrak{b}(|pD|)$.
\end{proof}

\begin{lemma} \label{multipli}
Let $Y$ be a smooth projective variety and let $D$ be a divisor on $Y$ such that $\kappa(X,D)\geq 0$. Let $p_0 \in \mathbb{N}(D)$ be such that $\mathbb{N}(D) \cap \{m : m \geq p_0\} = e\mathbb{N} \cap \{m : m \geq p_0\}$, where $e$ is the exponent of $D$. If there exists an effective divisor $\overline{N}$ on $Y$ such that $\mathcal{J}(Y,\|mp_0D\|)(-\overline{N}) \subseteq \mathfrak{b}(|mp_0D|)$ for every $m \geq 1$, then there exists an effective divisor $N$ such that $\mathcal{J}(Y,\|pD\|)(-N) \subseteq \mathfrak{b}(|pD|)$ for every $p  \geq p_0$, $ p \in \mathbb{N}(D)$.
\end{lemma}

\begin{proof}
First of all, for every $p \in \mathbb{N}(D)$, choose $D_p \in |pD|$. By definition of $e=e(D)$ there exists an $s \in \mathbb{N}$ such that $p_0=se$. When $p \geq p_0$ write $p$ as $mp_0+p_0+re$, with $0 \leq r < s$ and a certain $m \geq 0$. We have that $\mathcal{J}(Y, \|pD\|) \subseteq \mathcal{J}(Y, \|mp_0D\|)$ by \cite[Theorem 11.1.8(ii)]{LazII}. 
Moreover, by hypothesis, $\mathcal{J}(Y, \|mp_0D\|)(-\overline{N}) \subseteq \mathfrak{b}(|mp_0D|)$ and, for every $0 \leq r < s$, $\mathcal{O}_Y(- D_{p_0+re}) \subseteq \mathfrak{b}(|(p_0+re)D|)$. Consider also that $ \mathfrak{b}(|mp_0D|) \cdot \mathfrak{b}(|(p_0+re)D|) \subseteq \mathfrak{b}(|(mp_0+p_0+re)D|)=\mathfrak{b}(|pD|)$ (see \cite[Example 2.4.11]{LazII}). The thesis now follows by taking $$N:= \overline{N}+\sum_{r=0}^{s-1} D_{p_0+re}.$$
\end{proof}
From the two previous lemmas we immediately deduce the following.

\begin{corollary}\label{cor:stellahom}
Let $Y$ be a smooth projective variety and let $D$ be a divisor on $Y$ such that $\kappa(X,D)\geq 0$. The divisor $D$ verifies b-\textnormal{(}$\star$\textnormal{)} $\Leftrightarrow$ the divisor $mD$ verifies b-\textnormal{(}$\star$\textnormal{)}, for some $m\geq 1$.
\end{corollary}
\begin{lemma} \label{bordo}
Let $Y$ be a smooth projective variety and let $D$ be a divisor on $Y$ such that $\kappa(X,D) \geq 0$. Given any $\mathbb{Q}$-divisor $\Delta$ on $Y$ there exists an effective divisor $M$ on $Y$ such that $\mathcal{J}((Y,\Delta);\|pD\|)(-M) \subseteq \mathcal{J}(X,\|pD\|)$ for every $p$.
\end{lemma}

\begin{proof}
Take $M \geq -\Delta$ and such that $M$ is effective (we are not supposing that $\Delta \geq 0$). Then by \cite[Example 9.3.57(i)]{LazII} $\mathcal{J}((X,\Delta);\|pD\|)(-M)=\mathcal{J}(X, \Delta+ \|pD\|)(-M)$. Moreover the latter, by \cite[Proposition 9.2.31]{LazII}, is just $\mathcal{J}(X, \Delta+M+\|pD\|)$. We can conclude noticing that, since $M + \Delta \geq 0$, $\mathcal{J}(X, \Delta+M+\|pD\|) \subseteq \mathcal{J}(X,\|pD\|)$ by \cite[Proposition 9.2.32(i)]{LazII} .
\end{proof}

\begin{remark} \label{bordormk}
Lemma \ref{bordo} means, in particular, that if equation (\ref{fondamentale}) holds on a smooth variety $Y$ then, summing $M$ to $N$, the same equation holds for the pair $(Y,\Delta)$, for any $\Delta$.
 \hfill $\square$ \end{remark}

Given a normal projective $X$ and a Cartier divisor $D$, even if we know that (\ref{fondamentale}) holds on a birational modification of $X$, we cannot conclude that it holds on $X$ itself. Anyway given the previous corollary and remark, we can state the following:

\begin{corollary} \label{coppia}
Let $X$ be a normal projective variety and let $D$ be a divisor on $X$ satisfying property \textnormal{b-($\star$)}. Suppose that there exists a Weil $\mathbb{Q}$-divisor $\Delta_X$ on $X$ such that $(X,\Delta_X)$ is a pair. Then there exists an effective Cartier divisor $N$ on $X$ such that $$\mathcal{J}((X,\Delta_X);\|pD\|)(-N) \subseteq \overline{\mathfrak{b}(|pD|)}$$ for every $p \in \mathbb{N}(D)$. 
\end{corollary}

\begin{proof}
By definition there exists a smooth projective variety $Y$, an effective divisor $\overline{N}$ on $Y$ and a birational morphism $\mu: Y \rightarrow X$ such that $\mu^*(pD)$ satisfies (\ref{fondamentale}): $\mathcal{J}(Y,\|\mu^*(pD)\|)(-\overline{N}) \subseteq \mathfrak{b}(|\mu^*(pD)|)$ for every $p \in \mathbb{N}(D)$. Define the divisor $\Delta_Y$ on $Y$ so that $\mu^*(K_X+\Delta_X) \equiv K_Y+\Delta_Y$ and $\mu_*(\Delta_Y)=\Delta_X$. (For example, if $X$ is smooth and $\Delta_X=0$ then $\Delta_Y:=-K_{Y/X}$).
Applying Remark \ref{bordormk} and taking an $\tilde{N}$ we know that, for every $p \in \mathbb{N}(D)$, $\mathcal{J}((Y,\Delta_Y); \|\mu^*(pD)\|)(-\tilde{N}) \subseteq \mathfrak{b}(|\mu^*(pD)|)$. By Lemma \ref{dasotto} we can assume that $\tilde{N}$ is $\mu^*(N)$, for a certain Cartier effective divisor $N$ on $X$.
At this point just apply projection formula, birational transformation rule (see \cite[Proposition 9.3.62]{LazII} or - for asymptotic multiplier ideals - \cite[Lemma 2.18]{LS}) and Lemma \ref{chiusura integrale} (just recall that $\mathfrak{b}(|\mu^*(pD)|)=\mathfrak{b}(|pD|)\cdot \mathcal{O}_Y$).
\end{proof}

If $X$ is smooth and $\Delta_X=0$, considering subadditivity and Brian\c con--Skoda's theorem, we can say something more:

\begin{corollary} \label{skoda}
Let $X$ be a smooth projective variety of dimension $n$ and let $D$ be a divisor that verifies property \textnormal{b-$(\star)$}. Then there exists an effective divisor $M$ on $X$ such that $$\mathcal{J}(X,\|npD\|)(-M) \subseteq \mathfrak{b}(|pD|)$$ for every $p \in \mathbb{N}(D)$.
\end{corollary}
\begin{proof}
Consider the outcome of Corollary \ref{coppia} and the divisor $N$. By Brian\c con--Skoda's theorem (see \cite[Theorem 9.6.26 and Example 9.6.5]{LazII}) $\overline{\mathfrak{b}(|pD|)}^n \subseteq \mathfrak{b}(|pD|)$, hence $\mathcal{J}(X,\|pD\|)^n(-nN) \subseteq \mathfrak{b}(|pD|)$. To conclude just set $M:=nN$ and notice that, by Demailly, Ein, Lazarsfeld's subadditivity theorem (see \cite[Corollary 11.2.4]{LazII}), $\mathcal{J}(X, \|npD\|) \subseteq \mathcal{J}(X,\|pD\|)^n$.
\end{proof}

\begin{lemma} \label{somma}
Let $X$ be a normal projective variety and let $D_1$,$D_2$ be two divisors on $X$ verifying \textnormal{b-($\star$)}. Suppose that, given  $m \in \mathbb{N}$, for every $k$, $|km(D_1+D_2)|=|kmD_1|+|kmD_2|$. Then $D_1+D_2$ verifies \textnormal{b-($\star$)} as well.
\end{lemma}
\begin{proof}
By definition, 
for $i=1,2$ there exist a smooth projective variety $Y_i$, a birational map $\mu_i: Y_i \rightarrow X$ and an effective divisor $N_i$ on $Y_i$ such that $\mathcal{J}(Y_i, \|\mu_i^*(pD_i)\|)(-N_i) \subseteq \mathfrak{b}(|\mu_i^*(pD_i)|)$ for every $p \in \mathbb{N}(D_i)$. Take a smooth projective variety $Y$ that birationally dominates $Y_i$, i.e., for $i=1,2$ we have a birational map $\sigma_i: Y \rightarrow Y_i$.
By Lemma \ref{quasibig}, for $i=1,2$,  for every $p \in \mathbb{N}(D_1) \cap \mathbb{N}(D_2)$, $$\mathcal{J}(Y, \|\sigma_i^*(\mu_i^*(pD_i))\|)(-\sigma_i^*(N_i)) \subseteq \mathfrak{b}(|\sigma_i(\mu_i^*(pD_i))|).$$ Moreover by Demailly, Ein, Lazarsfeld's subadditivity theorem (see \cite[Example 9.5.24]{LazII}), and for every sufficiently large $p=km$, we have that $$\mathcal{J}(Y, \|\sigma_1^*(\mu_1^*(pD_1)) + \sigma_2^*(\mu_2^*(pD_2))\|) \subseteq \mathcal{J}(Y, \|\sigma_1^*(\mu_1^*(pD_1))\|) \cdot \mathcal{J}(Y, \|\sigma_2^*(\mu_2^*(pD_2))\|).$$ Since we also always have the opposite inequality for  base ideals, that is $\mathfrak{b}(|\sigma_1^*(\mu_1^*(pD_1))|) \cdot \mathfrak{b}(|\sigma_2^*(\mu_2^*(pD_2))|) \subseteq \mathfrak{b}(|\sigma_1^*(\mu_1^*(pD_1)) + \sigma_2^*(\mu_2^*(pD_2))|)$, then we can easily conclude by Corollary \ref{cor:stellahom}. 
\end{proof}

\begin{corollary} \label{aggiungereeccezionali}
Let $X$ be a normal projective variety and let $D,F$ be two Cartier divisors on $X$, with $F$ effective and $\kappa(X,D) \geq 0$. Suppose that, given $m \in \mathbb{N}$, for every $k$ we have that $|k(mD)|=|k(mD-mF)|+kmF$. If $D-F$ verifies \textnormal{b-($\star$)} then $D$ verifies \textnormal{b-($\star$)} as well.
\end{corollary}
\subsection{Examples of divisors verifying property \textnormal{b-($\star$)}}  We will analyze preimages of big divisors and finitely generated divisors. Then we will deal with surfaces and threefolds.

\begin{proposition} \label{goodmoriwaki}
Let $X', X$ be normal projective varieties, $D$ a Cartier divisor on $X'$, $f: X' \rightarrow X$ a surjective morphism and let $B$ be a $\mathbb{Q}$-Cartier big divisor on $X$. Suppose we have that $D \sim_{\mathbb{Q}} f^*(B)$. Then $D$ verifies \textnormal{b-($\star$)}.
\end{proposition}

\begin{proof}
Taking Stein factorization, since the preimage of a big divisor under a finite morphism is still big, we can assume that $f$ has connected fibers. There exists $p_0 \geq 0$ such that $p_0B$ is integral and $p_0D \sim_{\mathbb{Z}} f^*(p_0B)$. Since $p_0B$ is big then in particular verifies b-($\star$), hence, by Corollary \ref{prepinco}, $p_0D$ verifies b-($\star$). Hence also $D$ verifies b-($\star$) by Corollary \ref{cor:stellahom}.
\end{proof}

Recall that a divisor $D$ is said to be finitely generated if its $\C$-graded algebra of global sections $\oplus_{m\geq 0} H^0(X,mD)$ is finitely generated. A characterization which is relevant for our purposes is the following. 
\begin{lemma}\label{lem:fgcarat}
Let $X$ be a normal projective variety. A Cartier divisor $D$ on $X$ is finitely generated if, and only if, there exists an integer $p_0 \in \mathbb{N}(D)$ such that $\mathfrak{b}(|mp_0D|)=\mathfrak{b}(|p_0D|)^m$ for every $m \in \mathbb{N}$. 

\end{lemma}
\begin{proof}We may assume $\kappa(X,D) \geq 0$. One implication is obvious. 
Let us prove the other. If $\mathfrak{b}(|mp_0D|)=\mathfrak{b}(|p_0D|)^m$ then let $f:Y \rightarrow X$ be a resolution of the base ideal of $|p_0D|$, with $Y$ normal, so that $\mathfrak{b}(|f^*(p_0D)|)=\mathcal{O}_Y(-F_0)$ and $|f^*(p_0D)|=|M_0|+F_0$, where $|M_0|$ is base point free. Then $\mathfrak{b}(|f^*(mp_0D)|)=\mathfrak{b}(|mp_0D|)\cdot \mathcal{O}_Y=\mathfrak{b}(|p_0D|)^m \cdot \mathcal{O}_Y= (\mathfrak{b}(|p_0D|) \cdot \mathcal{O}_Y)^m=\mathfrak{b}(|f^*(p_0D)|)^m=\mathcal{O}_Y(-mF_0)$. Hence $|f^*(mp_0D)|=|M_m|+mF_0$ and thus clearly we have that $M_m \sim_{\mathbb{Z}} mM_0$. Since $H^0(X,mp_0D)=H^0(Y,f^*(mp_0D))=H^0(Y,mM_0)$, we see that $p_0D$ is finitely generated because $M_0$ is. 
Therefore also $D$ is finitely generated.
\end{proof}
\begin{lemma} \label{finitamentegenerato}
Let $X$ be a normal projective variety and let $D$ be a finitely generated Cartier divisor on $X$ such that $\kappa(X,D) \geq 0$. Then there exist  $p_0 \in \mathbb{N}(D)$ and a log resolution $\mu_0:Y \rightarrow X$ of $|p_0D|$ such that:
\begin{enumerate}[(a)]
\item $\mathbb{N}(D) \cap \{m : m \geq p_0\} = e\mathbb{N} \cap \{m : m \geq p_0\}$, where $e=e(D)$; \label{zero}
\item $\mu_0^*(|p_0D|) = |M_0|+F_0$, where $M_0$ is an effective divisor such that $|M_0|$ is base point free and $F_0$ is an effective divisor; \label{one}
\item $\mu_0$ is a log resolution of $|mp_0D|$ for every $m \in \mathbb{N}$; \label{two}
\item for every $m \in \mathbb{N}$, $\mu_0^*(|mp_0D|)=|M_m|+mF_0$, where $M_m \sim mM_0$ is an effective divisor and $|M_m|$ is base point free.\label{three}
\end{enumerate}
\end{lemma}

\begin{proof} 
Since $D$ is finitely generated there exists $p_0 \in \mathbb{N}(D)$  such that $\mathfrak{b}(|mp_0D|)=\mathfrak{b}(|p_0D|)^m$ for every $m \in \mathbb{N}$. Without loss of generality we can assume that $p_0$ is sufficiently large so that (\ref{zero}) is satisfied. Let $\mu_0: Y \rightarrow X$ be a log resolution of $\mathfrak{b}(|p_0D|)$, i.e., $\mu^*(|p_0D|)=|M_0|+F_0$, where $F_0$ is effective and $M_0$ is base point free. 
Clearly $\mu_0$, for every $m$, is a log resolution also for $\mathfrak{b}(|p_0D|)^m=\mathfrak{b}(|mp_0D|)$, hence $(\mathfrak{b}(|p_0D|) \cdot \mathcal{O}_Y)^m = \mathfrak{b}(|mp_0D|)  \cdot \mathcal{O}_Y$, that is: $\mu_0^*(|mp_0D|)=|M_m|+mF_0$, where $M_m$ must necessarily be linearly equivalent to $mM_0$.
\end{proof}


\begin{proposition} \label{contideali}
Let $X$ be a normal projective variety and let $D$ be a finitely generated Cartier divisor on $X$ such that $\kappa(X,D) \geq 0$. Then $D$ verifies \textnormal{b-($\star$)}.
\end{proposition}

\begin{proof}
Take $p_0$, $\mu_0$ and $F_0$ as in Lemma \ref{finitamentegenerato}.
By definition and for every $m$ and $k_m \gg 1$, $$\mathcal{J}(Y, \| \mu_0^*(mp_0D) \|)= \mathcal{J}\left(Y, \frac{1}{k_m}|k_m\mu_0^*(mp_0D)|\right).$$
By Lemma \ref{finitamentegenerato}, (\ref{two}) and (\ref{three}), the latter is just $\mathcal{O}_Y(-mF_0) $.

As for the base ideals, again by Lemma \ref{finitamentegenerato}, (\ref{two}) and (\ref{three}), we have that $\mathfrak{b}(|\mu_0^*(mpD)|) =\mathcal{O}_Y(-mF_0)$ for every $m$. Hence $\mathcal{J}(Y,\|\mu^*(mp_0D)\|)=\mathfrak{b}(|mp_0D|)$ for every $m \geq 1$. At this point we can conclude, as in Proposition \ref{goodmoriwaki}, by Corollary \ref{cor:stellahom}.
\end{proof}

In particular Corollary \ref{coppia} and Corollary \ref{skoda} hold for finitely generated divisors.

It is also possible, given a certain divisor verifying b-($\star$), to construct other divisors verifying b-($\star$). For example, take a normal projective variety $X$ and a divisor $D$ on $X$ that satisfies property b-($\star$). Consider any birational morphism $\mu:Y \rightarrow X$, with $Y$ smooth, and an effective $\mu$-exceptional divisor $F$ on $Y$. Then for every sufficiently large and divisible $k$ we have $H^0(Y,\mu^*(kD)+kF)=H^0(X,kD)=H^0(Y,\mu^*(kD))$, where the first equality depends on the projection formula and \cite[Lemma 1-3-2-(3)]{KMM}. Therefore $|k(\mu^*(D)+F)|=|k\mu^*(D)|+kF$ and thus we see that $\mu^*(D)+F$ verifies b-$(\star)$ by Corollary \ref{aggiungereeccezionali}.


\begin{remark}
Notice  that  multiplier ideals associated to divisors verifying b-($\star$) may not be  well-behaved under certain viewpoints. In particular while for a big divisor $B$ on a smooth variety $\cup_{p \in \mathbb{N}} \mathcal{Z}(\mathcal{J}(X,\|pB\|))$ has a nice geometric interpretation (namely it corresponds to the restricted base locus of $B$, see \cite[Corollary 2.10]{ELMNP}), the same does not hold for  divisors verifying b-($\star$). Just think to \cite[Example 1]{R}: $\mathbb{B}_-(D)=\emptyset$ since $D$ is nef, but  $\bigcup \mathcal{Z}(\mathcal{J}((X,\|pD\|)))=\supp(D)$. 
\hfill $\square$
\end{remark}

All divisors on surfaces satisfy property b-($\star$):

\begin{proposition} \label{surfaces}
Let $X$ be a normal projective surface and let $D$ be a Cartier divisor such that $\kappa(X,D) \geq 0$. Then $D$ verifies \textnormal{b-($\star$)}.
\end{proposition}
\begin{proof}
Possibly taking a desingularization of $X$ we can assume that $X$ is smooth. Then, by \cite[Theorem 14.19]{B}, either $D$ is finitely generated or $D$ is big. In both cases $D$ verifies b-($\star$).
\end{proof}

Also nef divisors on a threefold satisfy property b-($\star$). 

\begin{proposition} \label{prop:3folds}
 Let $X$ be a normal projective threefold and let $D$ be a Cartier divisor on $X$ such that $\kappa(X,D) \geq 0$. Either $D$ is finitely generated or $D$ is abundant.
 In particular, any nef Cartier divisor on $X$ such that $\kappa(X,D) \geq 0$ verifies \textnormal{b-($\star$)}.
 \end{proposition}
\begin{proof}
Let $\nu(D)$ be the numerical dimension of $D$ (see \cite{Nakayama}). 
Recall that $D$ is said to be abundant if $\nu(D)=\kappa(X,D)$ and that, in general, $\nu(D)\geq \kappa(X,D)$. If $\kappa(X,D) \leq 1$ then $D$ is finitely generated. 
If $\nu(D)=3$ then $D$ is big by \cite[Proposition V.2.7]{Nakayama} 
, and hence abundant. If $\kappa(X,D)=2=\nu(D)$ then $D$ is abundant by definition.

Let now $D$ be a nef divisor. If $D$ is finitely generated, then it verifies \textnormal{b-($\star$)} by Proposition \ref{contideali}.
Suppose then that $D$ is abundant. 
Recall that, given a normal projective variety $X$ and a nef and abundant divisor $D$ on $X$, by Kawamata's theorem (see \cite[Proposition 2.1]{Kawamata}), we have a diagram:
\begin{displaymath}
\xymatrix{ Y  \ar[d] ^{\mu} \ar[r]^f & Z \\
X &}
\end{displaymath}
which satisfies the following conditions:
\begin{enumerate}
\item $Y,Z$ are smooth projective varieties,
\item $\mu$ is a birational morphism,
\item $f$ is surjective and it has connected fibers,
\item there exists a nef and big $\mathbb{Q}$-divisor $B$  on $Z$ such that $\mu^*(D) \sim_{\mathbb{Q}} f^*(B)$.
\end{enumerate}

Therefore, by Proposition \ref{goodmoriwaki}, nef and abundant divisors are specific examples of divisors verifying property b-($\star$) and we are done. 
\end{proof}

\section{Volumes along subvarieties}
%

\subsection{Reduced volume}

 Let $X$ be a smooth projective variety, and let $V \subseteq X$ be a subvariety of dimension $d$. Let $L$ be a line bundle on $X$. We can define (see \cite{ELMNPrestricted}, \cite{PT}) the \emph{reduced volume of} $L$ \emph{along} $V$ as
 $$ \limsup_{m \rightarrow  + \infty} \frac{\dim(H^0(\mathcal{O}_V(mL) \otimes \mathcal{J}(\|mL\|)_V))}{m^d/d!},$$
 where $\mathcal{J}(\|mL\|)$ is the asymptotic multiplier ideal associated to $mL$ (see \cite[Definition 11.1.2]{LazII}). Recall also that if $\mathcal{J}$ is an ideal sheaf on $X$, the ideal $\mathcal{J} \cdot \mathcal{O}_V$ it is usually denoted by $\mathcal{J}_V$. 

\begin{remark}
In the introduction of \cite{PT} it is said that it is worth studying $\mu(V,L)$ even if a natural ring structure on $\bigoplus H^0(V, \mathcal{O}_V(mL) \otimes \mathcal{J}(\|mL\|)_V)$ is not known. Essentially this depends on the fact that the subadditivity inequality for asymptotic multiplier ideals goes into the ``wrong direction''. Anyway, if $L$ satisfies (\ref{fondamentale}) (i.e., $\mathcal{J}(\|mL\|) (-G) \subseteq \mathfrak{b}(\|mL\|)$) then $\bigoplus H^0(V, \mathcal{O}_V(mL) \otimes \mathcal{J}(\|mL\|)_V(-G_V))$ is actually a graded linear series by \cite[Theorem 11.1.8]{LazII}. \hfill $\square$
 \end{remark}

The aim of the next subsection is to understand the relationship between reduced volume and asymptotic intersection number, at least when the line bundle $L$ verifies property b-($\star$). Before doing that, let us recall the following two useful lemmas:

  \label{sect:volume}
\begin{lemma}[\cite{PT}, Lemma 3.5] \label{lem:ineq}
Let $X$ be a smooth projective variety, let $L$ be a line bundle on $X$ and let $V \subseteq X$ be a subvariety of $\dim V = d > 0$ such that 
$V \not\subseteq \mathbb{B}(L)$.
Then $\mu(V, L) \ge \| L^d \cdot V \|$.
\end{lemma}
\begin{lemma}[\cite{ELMNPrestricted}, Lemma 3.3] \label{lem:pk}
Let $X$ be a smooth projective variety. Let $V \subseteq X$ be a subvariety of $\dim V = k > 0$ and let $L$ be a line bundle on $X$. Then for every $h$ we have 
$$
\limsup_{m \to \infty}
		\frac{h^0(V, \cO_V(mL) \otimes \mathcal{J}(\|mL\|)_V)}{m^d/d!}=\limsup_{p \to \infty}
		\frac{h^0(V, \cO_V(phL) \otimes \mathcal{J}(\|phL\|)_V)}{h^kp^d/d!},
		$$
\end{lemma}

\subsection{Reduced volume for divisors that satisfy property \textnormal{b-($\star$)}}

\begin{proof} [Proof of Theorem \ref{thm:mu}]
By Lemma \ref{lem:ineq} it is sufficient to prove the inequality
$\| L^d\cdot V\|\geq \mu(V,L)$. 
Since $X$ is smooth and $L$ satisfies b-($\star$), then we can assume that there exists a birational morphism $\phi:\overline{X} \rightarrow X$, where $\overline{X}$ is smooth, such that $V \not \subseteq \phi(\mathrm{Exc}(\phi))$ and such that $\phi^*(L)$ satisfies (\ref{fondamentale}) for all $p$ sufficiently large. Then by \cite[Lemma 2.3]{PT}, $\mu(V,L)=\mu(V', \phi^*(L))$, where $V'$ is the strict transform of $V$. Since, by definition, also $\|L^d \cdot V\|=\|\phi^*(L)^d \cdot V'\|$ then we can, and we will, suppose that $L$ satisfies (\ref{fondamentale}) on $X$ and for every $p$ in $\mathbb{N}(L)$. In particular we have that there exists an effective divisor $G$ on $X$, that can be taken arbitrarily ample, such that, for every $m \in \mathbb{N}(L)$, for every $V \not \subseteq \Lambda:={\mathbb{B}(L) \cup \Supp(G)}$, 
\begin{equation} \label{eq:primariduzione}
\mathcal{J}(X,\|mL\|)_V(-G_{|V}) \subseteq {\mathfrak b}(|mL|)_V.
\end{equation}

At this point we argue as in \cite[Proof of Theorem 2.13]{ELMNPrestricted}.  Consider a common log-resolution of both $\mathfrak{b}(|mL|)$ and $\mathcal{J}(X,\|mL\|)$
$$
 \pi_m : X_m\to X.
$$
Write 
$$
 \mathfrak{b}(|mL|)\cdot \cO_{X_m}=\cO_{X_m}(-E_m)\textrm{ and } \mathcal{J}(X,\|mL\|)\cdot \cO_{X_m}=\cO_{X_m}(-F_m),
$$
and 
$$
 \pi_m^*(mL)=M_m+E_m=N_m+F_m.
$$
Notice that by construction $M_m$ is base point free, hence nef. Moreover, from the inclusion $ \mathfrak{b}(|mL|)\subseteq \mathcal{J}(X,\|mL\|)$ we deduce that
\begin{equation}\label{eq:NM}
N_m\geq M_m.
\end{equation}
 We denote by $V_m$ the strict transform of $V$ in $X_m.$ Then from   (\ref{eq:NM}) we deduce 
\begin{equation}\label{eq:6}
\vol({M_m}_{| V_m}+({\pi_m}_{| V_m})^*G_{|V})\geq \vol((N_m)_{| V_m})\geq \vol((M_m)_{| V_m}). 
\end{equation}
On the other hand we have 
\begin{align}\label{eq:sub}
h^0(V_m, k(N_m)_{| V_m}) &\geq  
h^0(V,\cO_V(kmL)\otimes\mathcal{J}(X,\|mL\|)^k_{V}) \\ \nonumber
& \geq 
h^0(V,\cO_V(kmL)\otimes\mathcal{J}(X,\|kmL\|)_{V})
\end{align}
where for the second inequality we make use of the subadditivity property of the asymptotic multiplier ideals (see \cite{DEL} or \cite[Theorem 11.2.3]{LazII}).
Multiplying by $d!/(km)^d$ and letting first $k$ and then $m$ go to infinity in (\ref{eq:sub}), we deduce that 
\begin{equation}\label{eq:7}
\limsup_{m\to \infty} \frac{\vol({N_m}_{|V_m})}{m^d}\geq \mu(V,L)
\end{equation}
(here we also use Lemma \ref{lem:pk}).
Therefore, by (\ref{eq:6}) and (\ref{eq:7}), it is sufficient to prove 
that 
\begin{align}\label{eq:suff}
 \lim_{m\to \infty} \frac{\vol(({M_m}_{|V_m}+({\pi_m}_{|V_m})^*G_{|V}))-\vol({M_m}_{|V_m})}{m^d}  & = & \\
 \limsup_{m\to \infty} \frac{\vol(({M_m}_{|V_m}+({\pi_m}_{|V_m})^*G_{|V}))-\vol({M_m}_{|V_m})}{m^d}&=0. \nonumber
\end{align}
Since the divisors appearing in the formula above are nef (recall that $G$ has been taken arbitrarily ample), then the volumes  are simply computed as intersection numbers. 
Notice that if $A$ is an ample divisors on $X$  such that $A-L$ is effective then $\pi^*(mA)-M_m\geq 0$. Therefore 
$$
 (M^{d-i}_m\cdot (\pi_m^*G)^i\cdot V_m)\leq 
 ((\pi^*(mA))^{d-i}\cdot (\pi_m^*G)^i\cdot V_m)
$$
which implies 
$$
  \lim_{m\to \infty} \frac{(M^{d-i}_m\cdot (\pi_m^*G)^i\cdot V_m)}{m^d}=0.
$$
Then (\ref{eq:suff}) is proved, and we are done. 
\end{proof}

\begin{proof}[Proof of Corollary \ref{cor:caratt}]
Suppose that $d \leq \kod(X,L) < \dim(X)$. Then we can always find a $d$-dimensional subvariety $V \subsetneq X$ such that the Iitaka fibration $f$ of $L$, restricted to $V$, gives a generically finite map of degree $\delta > 1$. Hence by \cite[Theorem 1.1(1)]{PT} $\vol_{X|V}(L) >0$ and thus by Theorem \ref{thm:fujita} $\|L^d \cdot V\|= \delta \vol_{X|V}(L)$. But this is not possible since $\mu(V,L) \geq \|L^d \cdot V\|$ by Lemma \ref{lem:ineq}.
\end{proof}

\begin{remark}
The proof of Theorem \ref{thm:mu} is based on the comparison between the base ideals $\mathfrak{b}(|mL|)$ and the multiplier ideals $\mathcal{J}(X,\|mL\|)$ on appropriate resolutions. Therefore all comes down to valuations: instead of requiring $L$ to verify property b-($\star$) we could just ask that $v(\mathcal{J}(X, \|mL\|))+v(G) \geq v(\mathfrak{b}(\|mL\|))$ for every divisorial valuation $v$ on $X$ and for every $m \in \mathbb{N}(L)$.
 \hfill $\square$ \end{remark}

In the proof of Theorem \ref{thm:mu} we are not actually using that $L$ verifies property b-($\star$), but rather Equation (\ref{eq:primariduzione}), i.e., that $L$ verifies ($\star$) on $V$. We are then naturally lead to the following more general situation.
Given a big graded linear series $W_\bullet$  on a smooth variety $V$, associated to a divisor $D_V$, we wonder if it is true that there exists an effective divisor $G_V$ on $V$ such that for all $m$ for which $W_m \not = 0$ then 
\begin{equation}\label{eq:stellaV}
\mathcal{J}(\|W_m\|)(-G_V) \subseteq \mathfrak{b}(W_m)
\end{equation}
 (see \cite[Definition 11.1.24]{LazII} for the definition of asymptotic multiplier ideals for graded series). \\
If  $V$ is a subvariety of a smooth projective variety $X$, $L\in \Pic(X)$ and the restricted linear series $W_m:=H^0(X|V,mL)$ satisfies  (\ref{eq:stellaV}) we say that $L$ satisfies property ($\star$)$_V$. The proof of Theorem \ref{thm:mu} yields the following.

\begin{theorem} Let $X$ be a smooth projective variety and $L\in \Pic(X)$ such that $\kod(X,L)\geq 0$. Let $V\subseteq X$ be a smooth irreducible subvariety of dimension $d$. For every integer $m>0$ set $W_m:=H^0(X|V,mL)$. Suppose that
\begin{enumerate}
\item $L$ satisfies property  \textnormal{($\star$)$_V$};
\item $\mathcal{J}(V,\|W_m\|)=\mathcal{J}(X,\|mL\|)_V$.
\end{enumerate}
Then $\mu(V,L)=\|L^d\cdot V\|$.
\end{theorem}
Notice that hypothesis (2) in the theorem above is generically verified. 
Precisely, by induction on the codimension of $V$ and argueing as in \cite[\S 9.5.A]{LazII}, one can prove the following.  
\begin{proposition} \label{prop:generale}
Let $X$ be a smooth projective variety and let $D$ be a divisor on $X$ of non-negative Iitaka dimension. Let $c \in \mathbb{N}$, $ 1 \leq c \leq \dim(X)-1$. Let $W_1, \ldots, W_c$ be base point free linear series, maybe not complete. For every $j=1, \ldots, c$ let $H_j \in W_j$ be a general element. Set $V:=H_1 \cap \ldots \cap H_c$. Then $\mathcal{J}(V,\|D\|_V)=\mathcal{J}(X,\|D\|)_V$.
\end{proposition}

\section{Property \textnormal{b-($\star$)} and valuations} \label{nefabundant}
%
We have seen that a rather large class of divisors verifies (\ref{fondamentale}). It is then natural to ask if all integral divisors,  with non-negative Iitaka dimension, verify (\ref{fondamentale}). 
In this section we will prove two properties (see Proposition \ref{bound} and Proposition \ref{prop:controes}) of divisors that verify (\ref{fondamentale}) that may be helpful to look for counterexamples. 

We recall two definitions: 

\begin{definition}
Let $X$ be a normal projective variety and let $D$ be a Cartier divisor on $X$ with $\kappa(X,D)\geq 0$. Let $v$ be a discrete valuation on $K(X)$. We say that $D$ is \emph{$v$-bounded} if there exists a positive constant $C_v$ such that $v(|pD|) \leq C_v$ for every $p \in \mathbb{N}(D)$.
 \hfill $\square$ \end{definition}

\begin{definition} \label{asymptotic} (cf. \cite[Definition 2.2]{ELMNP}). 
Let $X$ be a normal projective variety and let $D$ be a Cartier divisor such that $\kappa(X,D)\geq0$. Let $v$ be a divisorial valuation on $X$ (as in \cite[\S 2]{DH}). 
Let $e=e(D)$ be the exponent of $D$. The \textit{asymptotic order of vanishing} of $D$ along $v$ is 
$$v(\| D\|):=\lim_{p \rightarrow \infty} \frac{v(|peD|)}{pe},$$ 
where, for $p$ large enough such that $|peD|\not=\emptyset$, we set $v(|peD|)=v(D')$, where $D'$ is general in $|peD|$. Analogously $v(|peD|)=v(\mathfrak{b}(|peD|))$ where $\mathfrak{b}(|peD|)$ is the base ideal of the non-empty linear system $|peD|$. Actually the limit in the definition coincides with the inf.
 \hfill $\square$ \end{definition}

In particular, we prove below that if a divisor $D$ verifies (\ref{fondamentale}) then, given a divisorial valuation $v$, $v(\|D\|)=0 \Rightarrow D$ is $v$-bounded. Then if, for example, we could find a smooth variety $X$, an effective integral divisor $D$, and a point $x \in X$ such that $\ord_x(|mD|)$ goes to infinity but less than linearly (i.e., $\lim_{m \rightarrow + \infty} \frac{\ord_x(|mD|)}{m} = 0$) then this $D$ would not verify (\ref{fondamentale}).

Recall that 
in general it is an open problem to establish if $v(\| D \|)=0$ implies that $D$ is $v$-bounded.  
However it is already known to be true when $X$ is a smooth projective curve or a normal projective surface (see \cite[Remark 5.44]{Salvo}) 
or when $D$ is a big divisor on an arbitrary normal projective variety (see \cite[Proposition 2.8]{ELMNP}). 
Following \cite{ELMNP}, with a small  further effort we can prove $v$-boundedness for divisors verifying property b-($\star$). 

\begin{proposition} \label{bound}
Let $X$ be a normal projective variety and let $D$ be a divisor on $X$ satisfying property \textnormal{b-($\star$)}. Then there exists a Cartier divisor $N$ on $X$ such that, for every divisorial valuation $v$ on $\mathbb{C}(X)$, $v(\|D\|)=0 \Rightarrow v(|pD|) \leq v(N)$ for every $p \in \mathbb{N}(D)$. In particular for any $v$ such that $v(\|D\|)=0$, $D$ is $v$-bounded.
\end{proposition}
\begin{proof}
By definition, there exists a smooth projective variety $Y$, an effective divisor $\overline{N}$ on $Y$ and a birational morphism $\mu: Y \rightarrow X$ such that $\mu^*(pD)$ satisfies (\ref{fondamentale}): $\mathcal{J}(Y,\|\mu^*(pD)\|)(-\overline{N}) \subseteq \mathfrak{b}(|\mu^*(pD)|)$ for every $p \in \mathbb{N}(D)$.
First of all, by Lemma \ref{dasotto}, replace $\overline{N}$ by $\mu^*(N)$, with $N$ effective and Cartier on $X$. Given a divisorial valuation $v$, by the hypothesis we have that:
\begin{equation}  \label{e1}
v(\mathcal{J}(Y,\|\mu^*(pD)\|))+v(\mu^*(N)) \geq v(|\mu^*(pD)|), \text{ for every $p \in \mathbb{N}(D)$.}
\end{equation}
Moreover, since $Y$ is smooth, by \cite[Theorem 11.1.8(iv)]{LazII}, we also have that $\mathfrak{b}(|\mu^*(pD)|) \subseteq \mathcal{J}(Y, \|\mu^*(pD)\|)$, that is,  
\begin{equation}\label{e2}
v(|\mu^*(pD)|) \geq v(\mathcal{J}(Y, \|\mu^*(pD)\|)), \text{ for every $p \in \mathbb{N}(D)$.}
\end{equation}
By Demailly, Ein, Lazarsfeld's subadditivity theorem, exactly as in \cite[Section 2]{ELMNP}, we have that $$\lim_{p \rightarrow \infty} \frac{v(\mathcal{J}(Y, \|\mu^*(pD)\|))}{p} = \sup_{p} \frac{v(\mathcal{J}(Y, \|\mu^*(pD)\|))}{p},$$
hence, dividing (\ref{e2}) by $p$ and taking limits, we see that $v(\|D\|)=v(\|\mu^*(D)\|)=0$ forces $v(\mathcal{J}(Y, \|\mu^*(pD)\|))$ to be $0$ for every $p \in \mathbb{N}(D)$. By (\ref{e1}) we are done. 
\end{proof}

Another property of divisors verifying ($\star$) in relation with valuations is the following.
\begin{proposition} \label{prop:controes}
Let $X$ be smooth and let $D$ be an integral divisor such that $\kappa(X,D) \geq 0$. If $D$ satisfies \textnormal{(\ref{fondamentale})}, then for every $v$ divisorial valuation  we have $\sup_{p \in \mathbb{N}(D)} \left\{\frac{v(\mathcal{J}(X, \|pD\|))}{p}\right\} = v(\|D\|)$. 
\end{proposition}
\begin{proof}
For every $p \in \mathbb{N}(D)$ $$\frac{v(\mathcal{J}(X,\|pD\|))}{p} + \frac{v(N)}{p} \geq \frac{v(|pD|)}{p}$$ and $$\ \frac{v(\mathcal{J}(X,\|pD\|))}{p} \leq \frac{v(|pD|)}{p}.$$ 
As in the proof of the previous proposition, taking limits we can conclude.
\end{proof}

\bibliographystyle{plain}  
\bibliography{multiplierbaseideal}

\end{document}